\documentclass[11pt]{article}
\usepackage{graphicx}

\usepackage{amsmath,amssymb,amsbsy,amsfonts,amsthm,latexsym,amsopn,amstext,amsxtra,euscript,amscd,color,breqn,mathrsfs,array}

\usepackage[capbesideposition=outside,capbesidesep=quad]{floatrow}

\restylefloat{table}
            
\usepackage{multirow}

\usepackage[justification=centering]{caption}

\usepackage{longtable}

\usepackage{xcolor}
\usepackage{tikz}
\usetikzlibrary{shapes}
\usetikzlibrary{decorations.markings}

 \usepackage[space]{grffile} % For spaces in paths
 \usepackage{etoolbox} % For spaces in paths
 \usepackage{diagbox} % Diagonal split in tables
 \makeatletter % For spaces in paths
 \patchcmd\Gread@eps{\@inputcheck#1 }{\@inputcheck"#1"\relax}{}{}
 \makeatother

\newcommand\blfootnote[1]{%
  \begingroup
  \renewcommand\thefootnote{}\footnote{#1}%
  \addtocounter{footnote}{-1}%
  \endgroup
}            

\newtheorem{theorem}{Theorem}[section]

\newtheorem{definition}[theorem]{Definition}
\newtheorem{corollary}[theorem]{Corollary}
\newtheorem{lemma}[theorem]{Lemma}
\newtheorem{proposition}[theorem]{Proposition}

\newtheorem{remark}[theorem]{Remark}
\newtheorem{observation}[theorem]{Observation}

\def\+{\oplus}

\def\F{{\mathbb F}}
\def\Z{{\mathbb Z}}

\def\Z{{\mathbb Z}}

\def\wt{{\rm wt}}

\def\00{{\bf 0}}
\def\11{{\bf 1}}
\def\+{\oplus}

\def\Mod{{\ \rm Mod\ }}

\def\\{\cr}
\def\({\left(}
\def\){\right)}

\def\wt{{\rm wt}}

\providecommand{\newoperator}[3]{%
  \newcommand*{#1}{\mathop{#2}#3}}

\newoperator{\FD}{\mathrm{FD}}{\nolimits}

\newcommand{\ps}{\mathcal{P}}

\usepackage[a4paper]{geometry}

\begin{document}

\title{\bf An infinite family of 0-APN monomials with two parameters}
\author{Nikolay Kaleyski$^1$, Kjetil Nesheim$^1$, Pantelimon St\u anic\u a$^2$\\
\small $^1$ Department of Informatics, University of Bergen,\\
\small 5020, Bergen, Norway;\\
\small  {\tt Nikolay.Kaleyski@uib.no,kjetil.nesheim@protonmail.com}\\
\small $^2$ Department of Applied Mathematics, Naval Postgraduate School\\
\small Monterey, CA 93943-5212, U.S.A.; {\tt pstanica@nps.edu}
}
\date{}

\maketitle
\thispagestyle{empty}

\abstract{We consider an infinite family of exponents $e(l,k)$ with two parameters, $l$ and $k$, and derive sufficient conditions for $e(l,k)$ to be 0-APN over $\F_{2^n}$. These conditions allow us to generate, for each choice of $l$ and $k$, an infinite list of dimensions $n$ where $x^{e(l,k)}$ is 0-APN much more efficiently than in general. We observe that the Gold and Inverse exponents, as well as the inverses of the Gold exponents can be expressed in the form $e(l,k)$ for suitable $l$ and $k$. We characterize all cases in which $e(l,k)$ can be cyclotomic equivalent to a representative from the Gold, Kasami, Welch, Niho, and Inverse families of exponents. We characterize when $e(l,k)$ can lie in the same cyclotomic coset as the Dobbertin exponent (without considering inverses) and provide computational data showing that the Dobbertin inverse is never equivalent to $e(l,k)$. We computationally test the APN-ness of $e(l,k)$ for small values of $l$ and $k$ over $\F_{2^n}$ for $n \le 100$, and sketch the limits to which such tests can be performed using currently available technology. We conclude that there are no APN monomials among the tested functions, outside of the known classes.}

\section{Introduction}
\label{secIntroduction}

\blfootnote{Some of the results in this paper were partially presented at Boolean Functions and Their Applications (BFA) 2022. In particular, all results from Section 4 onwards are completely new.}

We consider  vectorial Boolean functions, i.e. mappings over the vector space $\F_2^n$ or, equivalently, the finite field $\F_{2^n}$, where $n$ is some positive integer. The differential uniformity is one of the most important properties of vectorial Boolean functions from a cryptographic point of view since it measures their resistance to attacks such as differential cryptanalysis. More precisely, the differential uniformity $\Delta_F$ of a vectorial Boolean function $F$ is desired to be as low as possible. It is simple to see that $\Delta_F \ge 2$ for any vectorial Boolean function $F$. The best possible functions are thus the ones with $\Delta_F = 2$, which are called almost perfect nonlinear (APN). These functions are of interest since they also correspond to optimal objects and constructions in other fields of mathematics and computer science, including combinatorics, algebra, and coding theory. For instance, APN functions can be related to linear codes with prescribed parameters \cite{carlet1998codes}. The study of APN and PN functions, including finding new instances of such functions and investigating their properties, is thus interesting and relevant from multiple points of view.

Unfortunately, APN functions are generally difficult to find and analyze. To date, we know many instances of APN functions (see e.g.~\cite{yu2014matrix},~\cite{beierle2020new}) but very little can be said about their structure in general. We refer the reader to~\cite{carlet2021boolean,CS17} for a detailed overview of cryptographic Boolean functions.

Some of the oldest known instances of APN functions are monomials, or power functions, i.e. functions that can be expressed as polynomials of the form $F(x) = x^d$ over $\F_{2^n}$ for some positive integer $d$. Due to their relatively simple structure, these are some of the most studied and best understood vectorial Boolean functions, although the area is still riddled with open questions and unsolved problems.

Due to the general difficulty of constructing and analyzing APN functions, weaker notions of APN-ness have been introduced, such as that of partial APN-ness (pAPN-ness)~\cite{budaghyan2019partially} which we focus on in this paper. This means that any APN function is pAPN, but not necessarily vice-versa; and so this weaker notion can be used as a ``stepping stone'' in formulating constructions of APN functions and analyzing their properties. In particular, one of the motivations behind the notion of partial APN-ness is the possibility of learning more about the structure of APN permutations. We note that the existence of APN permutations over $\F_{2^n}$ with even $n$ (typically referred to as the ``big APN problem'') is one of the oldest and most important questions in the field of cryptographic Boolean functions; and that, while APN permutations over $\F_{2^n}$ with odd $n$ are known (for instance, all monomial APN functions are of this form), we still know very few examples and constructions of such functions.

In addition to the ``big APN problem'', one of the most important open questions in the area is the existence of APN monomials inequivalent to the six known families~\cite{dobbertin2001almostNewCase}. In the aforementioned paper, it is conjectured that these six known families exhaust all possible cases up to equivalence. According to~\cite{carlet2021boolean}, this conjecture has been computationally verified over $\F_{2^n}$ for all $n$ up to $34$, and also up to $42$ in the case of even $n$. Computationally searching for new APN monomials becomes very difficult for large values of $n$, since not only does the verification of the APN property require more effort, but the number of exponents that need to be checked grows exponentially with $n$. Finding constructions of $0$-APN monomials can thus also be useful for approaching this conjecture, since instead of examining all exponents over $\F_{2^n}$, only a smaller (and more promising set) has to be considered.

While certainly more tractable than APN functions, the behavior of 0-APN monomials is far from trivial, too. In the case of monomials, it is known that $x^d$ is 0-APN over $\F_{2^n}$ for infinitely many dimensions $n$, and that the set of dimensions can be characterized by computing the factorization of the polynomial $x^d + (x+1)^d + 1$ over $\F_2[x]$ \cite{budaghyan2020partially}. This is, however, difficult to do theoretically, and computationally, it is only feasible for relatively small values of $d$; for instance, the \textit{Magma} algebra system \cite{bosma1997magma} that we use for most of our computations struggles with computing such a factorization already for $d \ge 10^7$. While this number may seem large, we recall that if APN exponents distinct from representatives of the known families exist, they must be over finite fields $\F_{2^n}$ with $n \ge 35$, which involves exploring exponents much larger than this.

In this paper, we define an infinite family of exponents $e(l,k)  = \sum_{j = 0}^{l-1} 2^{jk}$ with two parameters $l$ and $k$ (that can take any positive integers as values), and give sufficient conditions that $n$ has to satisfy in order for $x^{e(l,k)}$ to be 0-APN over $\F_{2^n}$. For every choice of $l$ and $k$, our conditions produce infinitely many dimensions $n$ for which the exponent is 0-APN. Furthermore, we discuss how, with the help of some very simple computations, we can characterize the set of all dimensions $n$ for which $x^{e(l,k)}$ is 0-APN. Generating dimensions that satisfy our conditions requires minimal computational effort and amounts to computing the greatest common divisors of some integers. This is possible even for very large exponents of the form $e(l,k)$, e.g. a list of 24242 dimensions $n$ between $1$ and $100000$ for which $e(100,100) \approx 10^{2980}$ is 0-APN can be computed in a few seconds on \textit{Magma}. One of the advantages of our construction is that it becomes very easy to find dimensions where $x^{e(l,k)}$ is 0-APN; and, as remarked above, only a bit of additional computation is needed in order to characterize the set of all such dimensions.

Another advantage is that the algebraic degree of $e(l,k)$ is easily predictable, e.g. $\deg(e(l,k)) = l$ when $\gcd(k,n) = 1$ and $l < n$, which allows us to characterize with relative ease when the exponents $e(l,k)$ are cyclotomic equivalent to representatives from the known monomial APN families. We mathematically treat all cases, except that of the Dobbertin inverse, where the computations are too technical: the method that we use in our inequivalence proofs depends on the application of Lemma \ref{lemUniqueExpansion} to a set of integers describing the binary decomposition of the exponent of the monomial. Part of the lemma's hypothesis is that all these integers are distinct modulo $n$, and this requires the degenerate cases when two or more of them are congruent modulo $n$ to be treated separately before the lemma can be applied. Handling these cases is conceptually straightforward, but quite lengthy and technical in practice, especially when the number of cases that need to be considered is large. In the case of the inverse Dobbertin exponent, there is a substantial blowup in the number of cases than need to be handled, which would require a lengthy proof spanning tens of pages or even more. Instead of supplying a heavy and technical proof of this form, we provide computational data for $n \le 200$ showing that $e(l,k)$ can never be equivalent to the inverse of the Dobbertin function except in trivially small dimensions. The range $n \le 200$ covers all dimensions where the problem of finding new APN monomials can be handled in practice using our current knowledge and resources. Furthermore, we note that the ``missing'' case of the Dobbertin inverse only concerns odd dimensions $n$ that are multiples of $5$, while the theoretical characterizations from the remaining theorems give a complete description of the equivalence to the known families for all other values of $n$.

We show that the Gold and Inverse APN functions can always be represented in the form $e(l,k)$ for suitable choices of $l$ and $k$. Moreover, via~\cite{Nyb94}, the inverse of the Gold function $x\mapsto x^{2^r+1}$ over $\F_{2^n}$, $n$ odd, $\gcd(n,r)=1$, is given by $e(2r, \frac{n+1}{2})=\sum_{i=0}^{\frac{n-1}{2}} 2^{2ir}$. Furthermore, we show that representatives from the remaining families are never cyclotomic equivalent to $e(l,k)$ except for small dimensions.

Finally, we consider the exponents $e(l,k)$ for small values of $l$ and $k$, and computationally check for which dimensions $n$ below $100$ they are APN. We do not find any new APN exponents, but we see that the computation load needed to test APN-ness grows very quickly with the size of the exponent $e(l,k)$, even more so than with the dimension $n$. We provide more detailed comments about our computational experiments in Section \ref{secComputations}.

Unfortunately, our current computational methods are insufficient to check APN-ness of the exponents in high dimensions. We leave the computational exploration of the $e(l,k)$ exponents as a problem for future work.

\section{Preliminaries}

Let $n$ be a natural number. We denote by $\F_{2^n}$ the finite field with $2^n$ elements, and by $\F_2^n$ the vector space of dimension $n$ over $\F_2$. The set of non-zero elements of $\F_{2^n}$ is denoted by $\F_{2^n}^*$. A \textbf{vectorial Boolean function}, or $(n,m)$-function, is any mapping from $\F_2^n$ to $\F_2^m$. We concentrate on the case $n = m$. Any $(n,n)$-function can be uniquely represented as a univariate polynomial over $\F_{2^n}$ of the form
\begin{equation*}
    F(x) = \sum_{i = 0}^{2^n-1} a_i x^i,
\end{equation*}
with $a_i \in \F_{2^n}$. This is called the \textbf{univariate representation} of $F : \F_{2^n} \rightarrow \F_{2^n}$. A \textbf{monomial} function, or power function, is any $(n,n)$-function with univariate representation $F(x) = x^d$ for some natural number $d$.
The \textbf{algebraic degree} of $F$, denoted $\deg(F)$, is the largest (Hamming) weight of $i$ (that is, $\wt(i)$, which is the number of nonzero bits in the binary representation of~$i$), where $a_i\neq 0$.

The \textbf{differential uniformity} of $F : \F_{2^n} \rightarrow \F_{2^n}$ is defined as
\begin{equation*}
    \Delta(F) = \max_{a \in \F_{2^n}^*, b \in \F_{2^n}} \# \left\{ x \in \F_{2^n} \mid F(a+x) + F(x) = b \right\},
\end{equation*}
that is, as the largest number of solutions $x$ to $F(a+x) + F(x) = b$ for any choice of $a,b$ with $a \ne 0$. The differential uniformity is a measurement of the resistance provided by the function to differential cryptanalysis~\cite{Biham1991}, and should be as low as possible. The number of solutions $x$ to any equation of the form $F(a+x) + F(x) = b$ is even, and so the optimal value of the differential uniformity is $2$. If $\Delta(F) = 2$, we say that $F$ is \textbf{almost perfect nonlinear (APN)}.

The large number of vectorial Boolean functions makes it necessary to consider e.g. APN functions up to some suitable notion of equivalence in order to reduce the number of instances that have to be treated. Such an equivalence relation should leave the differential uniformity invariant, and should be as general as possible (in the sense that its equivalence classes should be as large as possible) in order to leave a small number of representatives that have to be considered. At present, the most general known equivalence relation used in practice is Carlet-Charpin-Zinoviev (CCZ) equivalence. Two $(n,n)$-functions $F$ and $G$ are said to be \textbf{CCZ-equivalent} if there is an affine permutation $A$ of $\F_{2^n}^2$ mapping the graph $\Gamma_F = \{ (x,F(x)) : x \in \F_{2^n} \}$ of $F$ to the graph $\Gamma_G$ of $G$. CCZ-equivalence is, in general, difficult to test computationally.

In the particular case of monomial functions, however, CCZ-equivalence reduces to a much simpler notion of equivalence. We say that $F, G : \F_{2^n} \rightarrow \F_{2^n}$ with $F(x) = x^d$ and $G(x) = x^e$ are \textbf{cyclotomic equivalent} if there exists a natural number $a$ such that either $2^a d$ is in the cyclotomic coset of $e$ modulo $2^n-1$, i.e.
\begin{equation*}
    2^a \cdot d \equiv e \pmod{2^n-1},
\end{equation*}
or the inverse of $d$ modulo $2^n-1$ is in the cyclotomic coset of $e$, i.e.
\begin{equation*}
    2^a \cdot d^{-1} \equiv e \pmod{2^n-1},
\end{equation*}
provided of course that $\gcd(d, 2^n-1) = 1$ so that the inverse $d^{-1}$ exists. We know that two monomials are CCZ-equivalent if and only if they are cyclotomic equivalent~\cite{dempwolff2018ccz,yoshiara2012equivalences}. Testing cyclotomic equivalence, in contrast to the more general CCZ-equivalence, is quite simple, and amounts to checking whether a small number of modular equations hold.

For natural numbers $n,k$, we will denote by $n \Mod k$ the least positive residue of $n$ modulo $k$, e.g. $11 \Mod 3 = 2$.

At present, we know of six infinite families of APN monomials; these are summarized in Table~\ref{tableMonomials}. In~\cite{dobbertin2001almostNewCase}, it is conjectured that no other APN monomials exist over $\F_{2^n}$ up to cyclotomic equivalence. In~\cite{carlet2021boolean}, it is reported that this has been computationally verified for $n \le 34$, and for $n \le 42$ in the case of even $n$. Despite this, the question of whether the list in Table~\ref{tableMonomials} is exhaustive up to cyclotomic equivalence remains open, and is one of the oldest and hardest unresolved questions in the area of APN functions. It is sometimes referred to as \textbf{Dobbertin's conjecture}.
\begin{table}
\small
  \centering
  \caption{Known infinite families of APN power functions over $\mathbb{F}_{2^n}$}
  \begin{tabular}{|c|c|c|c|c|}
    \hline
    Family & Exponent & Conditions & Algebraic degree & Source \\
    \hline
    \hline
    Gold & $2^i + 1$ & $\gcd(i,n) = 1$ & $2$ & \cite{gold1968maximal,Nyb94} \\
    \hline
    Kasami & $2^{2i} - 2^i + 1$ & $\gcd(i,n) = 1$ & $i + 1$ & \cite{janwa1993hyperplane, kasami1971weight} \\
    \hline
    Welch & $2^t + 3$ & $n = 2t + 1$ & $3$ & \cite{dobbertin1999almostWelch} \\
    \hline
    \multirow{2}{*}{Niho} & $2^t + 2^{t/2} - 1$, $t$ even & \multirow{2}{*}{$n = 2t + 1$} & $(t + 2)/2$ & \multirow{2}{*}{\cite{dobbertin1999almostNiho}} \\
         & $2^t + 2^{(3t + 1)/2} - 1$, $t$ odd &             & $ t + 1$ & \\
    \hline
    Inverse & $2^{2t} - 1$ & $n = 2t + 1$ & $n - 1$ & \cite{beth1993almost,Nyb94} \\
    \hline
    Dobbertin & $2^{4i} + 2^{3i} + 2^{2i} + 2^i - 1$ & $n = 5i$ & $i + 3$ & \cite{dobbertin2001almostNewCase} \\
    \hline
  \end{tabular}
  \label{tableMonomials}
\end{table}

Finding new instances of APN functions is challenging, especially when APN-ness is combined with other desirable properties, e.g. being bijective or being a monomial function. For this reason, various weaker notions of APN-ness have been defined in the literature (see~\cite{budaghyan2020partially,CK17}, for example). Following~\cite{budaghyan2020partially}, we say that a function $F : \F_{2^n} \rightarrow \F_{2^n}$ is $\mathbf{x_0}$\textbf{-APN} for some $x_0 \in \F_{2^n}$ if any $y,z \in \F_{2^n}$ satisfying
\begin{equation*}
    F(x_0) + F(y) + F(z) + F(x_0 + y + z) = 0
\end{equation*}
necessarily satisfy $(x_0 + y)(x_0 + z)(y + z) = 0$.
It is straightforward to verify that a function is APN if and only if it is $x_0$-APN for all $x_0 \in \F_{2^n}$.

In the case of monomials, it is shown~\cite{budaghyan2019partially} that if a monomial is $x_0$-APN for some $x_0 \ne 0$, then it is also $x_1$-APN for any $x_1 \ne 0$, and that $1$-APN-ness implies $0$-APN-ness for monomials. This means that a monomial can be either: APN; 0-APN but not 1-APN; not 0-APN. In this sense, 0-APN-ness is a natural intermediate step towards constructions of APN monomials.

One potential strategy for approaching Dobbertin's conjecture is to describe constructions of 0-APN monomials over fields $\F_{2^n}$ of high dimension $n$. As outlined in the introduction, in this work we introduce an infinite family of exponents which are particularly tractable from the point of view of 0-APN-ness and cyclotomic equivalence to the known families.

\section{An infinite family of 0-APN exponents with two parameters}

In this section, we introduce the exponents $e(l,k)$ and provide sufficient conditions on $n$ in order for $x^{e(l,k)}$ to be $0$-APN over $\F_{2^n}$. We recall that any monomial $x^d$ is 0-APN over infinitely many dimensions $n$, but in general it can be difficult to characterize these dimensions $n$ without doing computations such as factorizing the polynomial $x^d + (x+1)^d + 1$ over $\F_2$ (see \cite{budaghyan2018partially}), which can be a computationally hard task for large values of $d$. The class of exponents $e(l,k)$ has the advantage of being significantly more tractable in this sense. As outlined in the introduction, we are able to find $24242$ dimensions $n$ for which $e(100,100) \approx 10^{2980}$ is $0$-APN in a few seconds on \textit{Magma}, while factorizing $x^d + (x+1)^d + 1$ is computationally infeasible already for $d \ge 10^7$.

The sufficient conditions can be formulated in two ways. In the proof of Theorem~\ref{thmMainTheorem}, we show that if the expression $F(x) + F(1+x) + F(1)$ for $F(x) = x^{e(l,k)}$ vanishes for some $x \in \F_{2^n}$, then $x$ is in $\F_{2^{\gcd(lk,n)}}$, or it satisfies $(x/(x+1))^{e(l-1,k)} = 1$. Consequently, if $\gcd(lk,n) = 1$ and $\gcd(e(l-1,k),2^n-1) = 1$, then both of these cases imply that $x$ is a trivial solution, i.e. $x \in \F_2$.

The second condition, viz. $\gcd(e(l-1,k),2^n-1) = 1$, can be replaced by requiring that $\gcd(jk,n) = 1$ for all $j \in \{ 2, \ldots, l \}$ as explained in the remark following the theorem. This ``cascading'' condition is less general in the sense that it is not satisfied by all dimensions $n$ for which $x^{e(l,k)}$ is 0-APN according to Theorem~\ref{thmMainTheorem}. Nonetheless, it is somewhat simpler to evaluate and allows us to easily construct an infinite sequence of dimensions $n$ over which $x^{e(l,k)}$ is 0-APN even more easily.

\begin{definition}
    Let $l,k$ be natural numbers. We define the exponent $e(l,k)$ as
    \begin{equation*}
        e(l,k) = \sum_{j = 0}^{l-1} 2^{jk}.
    \end{equation*}
\end{definition}

The exponent $e(l,k)$ can also be expressed as
\begin{equation*}
    e(l,k) = \frac{2^{lk} - 1}{2^k - 1}
\end{equation*}
from the formula for the sum of a geometric progression.

\begin{theorem}
\label{thmMainTheorem}
  Let $n,l,k$ be natural numbers such that $\gcd(kl,n) = 1$ and $\gcd(e(l-1,k),2^n-1) = 1$. Then $x^{e(l,k)}$ is $0$-APN over $\F_{2^n}$.
\end{theorem}
  \begin{proof}
    Denote $e = e(l,k)$. Suppose that $x \in \F_{2^n}$ satisfies $x^e + (x+1)^e + 1 = 0$. For natural numbers $a \le b$, let $[a,b] = \{ a,a+1,\dots,b \}$, and let $\ps{I}$ denote the power set of a discrete set $I$. Furthermore, let $x^{2^{kI}}$ denote $\prod_{i \in I} x^{2^{ki}}$. Then $x^e + (x+1)^e + 1 = 0$ can be written as
    \begin{equation}
      \label{eqDerivative}
    x^e + \sum_{I \in \ps{[0,l-1]}} x^{2^{kI}} + 1 = \sum_{\substack{I \in \ps{[0,l-1]} \\ I \ne \emptyset, [0,l-1]}} x^{2^{kI}} = 0.
    \end{equation}
    Raising this to the power $2^k$ yields
    \begin{equation*}
      \sum_{\substack{I \in \ps{[1,l]} \\ I \ne \emptyset, [1,l]}} x^{2^{kI}} = 0.
    \end{equation*}
    Summing the two expressions causes all terms $x^{2^{kI}}$ corresponding to subsets $I$ that contain neither $0$ nor $l$ to cancel out, leaving us with
    \begin{equation*}
      \sum_{\substack{I \in \{ x \cup \{0\} \mid x \in \ps{[1,l-1]}\} \\ I \ne [1,l-1]}} x^{2^{kI}} + \sum_{\substack{I \in \{x \cup \{l\} \mid x \in \ps[1, l-1]\} \\ I \ne [1,l-1]}} x^{2^{kI}} = 0.
    \end{equation*}
    This then becomes
    \begin{equation*}
      x \left( \sum_{\substack{I \in \ps{[1,l-1]} \\ I \ne [1,l-1]}} x^{2^{kI}} \right) + x^{2^{lk}} \left( \sum_{\substack{I \in \ps{[1,l-1]} \\ I \ne [1,l-1]}} x^{2^{kI}} \right) = \left(x + x^{2^{lk}} \right) \left( \sum_{\substack{I \in \ps{[1,l-1]} \\ I \ne [1,l-1]}} x^{2^{kI}} \right) = 0.
    \end{equation*}
    If $x + x^{2^{lk}} = 0$, then we must have $x \in \F_{2^{\gcd(n,lk)}}$. However, by assumption, $\gcd(n,lk) = 1$, and so $x \in \F_2$. If $x \ne x^{2^{lk}}$, then we must have
      \begin{equation*}
	\left( \sum_{\substack{I \in \ps{[1,l-1]} \\ I \ne [1,l-1]}} x^{2^{kI}} \right) = \left( \sum_{\substack{I \in \ps{[0,l-2]} \\ I \ne [0,l-2]}} x^{2^{kI}} \right)^{2^k} = 0,
      \end{equation*}
      instead. Comparing this with \eqref{eqDerivative}, we see that this is simply
      \begin{equation*}
	(x^{e(l-1,k)} + (x+1)^{e(l-1,k)})^{2^k} = 0,
      \end{equation*}
      and hence
      \begin{equation}
	\label{eqInductionStep}
	x^{e(l-1,k)} + (x+1)^{e(l-1,k)} = 0.
      \end{equation}
      Clearly $x \ne 0, 1$, and so the above implies $(\frac{x}{x+1})^{e(l-1,k)} = 1$. If the second condition of the hypothesis is satisfied, i.e. $\gcd(e(l-1,k),2^n-1) = 1$, then we immediately have $\frac{x}{x+1} = 1$, i.e. $x = x + 1$, which is impossible. Therefore, $x^{e(l,k)}$ is 0-APN.
  \end{proof}

\begin{remark}
  The proof above could have also been continued by adding \eqref{eqInductionStep} to its $2^k$-th power; this would have produced the same equation as if we had added the derivative $x^{e(l-1,k)} + (x+1)^{e(l-1,k)} + 1$ to its $2^k$-th power since the extra term $1$ would have canceled out. By induction on $l$, we would have obtained the condition that if $\gcd(ik,n) = 1$ for $i = 2, 3, \dots, l$, then $x^{e(l,k)}$ must be $0$-APN. We have tested these conditions computationally, and, as expected, we observed that the condition in the statement of Theorem~\textup{\ref{thmMainTheorem}} always produces a set of dimensions $n$ that subsumes those given by the alternative condition described in this remark. This is why we have formulated the theorem only in terms of this more general condition, but we state the second condition as a corollary. 
\end{remark}

\begin{corollary}
    \label{corCascade}
    Let $n,l,k$ be natural numbers. Then, if $x^{e(l,k)} + (x+1)^{e(l,k)} + 1$ vanishes for some $x \in \F_{2^n}$, we must have $x \in \F_{2^{\gcd(jk,n)}}$ for some $j \in \{ 2, 3, \ldots, l \}$. 
    
    In particular, if $\gcd(jk,n) = 1$ for all $j \in \{ 2, 3, \ldots, l \}$, then $x^{e(l,k)}$ is $0$-APN over $\F_{2^n}$.
\end{corollary}

\begin{remark}
  The sufficient conditions of Corollary~\textup{\ref{corCascade}} allow us to explicitly determine the set of dimensions $n$ such that $x^{e(l,k)}$ over   $\F_{2^n}$ is $0$-APN. We can see that for any choice of $l$ and $k$, there are only finitely many dimensions $m = \gcd(jk,n)$ such that $x^{e(l,k)} + (x+1)^{e(l,k)} + 1$ vanishes on $\F_{2^m}$ but on no proper subfield of $\F_{2^m}$. Consequently, $e(l,k)$ is a $0$-APN exponent over $\F_{2^n}$ for any $n$ that is not a multiple of one of these dimensions $m$.

  For instance, the exponent $e(3,2) = 21$ can only violate the $0$-APN-ness on $\F_{2^6}$, $\F_{2^4}$ or $\F_{2^2}$ (or any extension field thereof). Hence, $x^{21}$ is $0$-APN over $\F_{2^n}$ for any $n$ that is not divisible by $2$, $4$ and $6$. Furthermore, we can computationally verify that $x^{21}$ is $0$-APN over $\F_{2^2}$ and $\F_{2^4}$, and that it is not $0$-APN over $\F_{2^6}$. Thus, $x^{21}$ is $0$-APN over $\F_{2^n}$ whenever $n$ is not a multiple of $6$. We remark that a proof of the same fact is given for $x^{21}$ in~\textup{\cite{budaghyan2020partially}} using the factorization of $x^{21} + (x+1)^{21} + 1$. The framework described in this remark allows this proof to be easily generalized to any function of the form $x^{e(l,k)}$, and allows us to characterize the values of $n$ for which $x^{e(l,k)}$ is $0$-APN for large values of $d$ for which it is not computationally feasible to factor $x^d + (x+1)^d + 1$.
\end{remark}

The conditions $\gcd(kl,n) = 1$ and $\gcd( e(l-1,k), 2^n-1) = 1$ are sufficient for $x^{e(l,k)}$ to be $0$-APN over $\F_{2^n}$ but are not necessary in general. The same is true for the ``cascading'' conditions formulated in Corollary~\ref{corCascade}. In particular, we can observe that the particular statement of the corollary can never be applied to finite fields $\F_{2^n}$ of even extension degree $n$ since $\gcd(2k,n) = 2$, and this violates the conditions in the corollary whenever $l > 1$.

The conditions of the corollary can be refined for instance as follows. Let $d = e(l,k)$. If $\gcd(kl,n) = 2$, then we can see that $x^d + (x+1)^d + 1 = 0$ can only have $x \in \F_4$ as a root, and $x^d$ is not $0$-APN only in the case when $x \in \F_4 \setminus \F_2$. Clearly, this happens precisely when $3 \nmid d$. If the exponent is a multiple of $3$, therefore, the restriction $\gcd(kl,n) = 1$ can be relaxed to $\gcd(kl,n) \le 2$.

A similar approach can be applied in general for $\gcd(kl,n) = m$ by imposing the restriction that $x^d + (x+1)^d + 1$ does not vanish on $\F_{2^{\gcd(m,n)}}$. For a fixed $m$, this essentially means that $d \pmod{2^m-1}$ must be a $0$-APN exponent in $\F_{2^m}$.

A trivial case is when the exponent $d$ satisfies $d \pmod{2^m-1} \equiv 0$ for some $m > 2$ dividing $n$. When this happens, the function $x^d$ coincides with the indicator function $1_0(x) = x^{2^m-1}$ over $\F_{2^m}$, and so it is always $0$-APN but can never be an APN function (except if $m = 2$). Since the primary motivation for our study is the possibility of identifying new APN monomials, all such cases can be excluded from consideration.

\begin{remark}
  To see how discriminating the condition in Theorem~\textup{\ref{thmMainTheorem}} is, we can perform a simple computational experiment as follows: pick some values of $k$ and~$l$, and generate all dimensions $n$ in some range that satisfy the conditions in Theorem~\textup{\ref{thmMainTheorem}}; by computing the number of roots of $x^e + (x+1)^e + 1 = 0$ for $e = e(l,k)$, we check whether $x^e$ is $0$-APN over $\F_{2^n}$ for all $n$ in the range, then compare the two sets. Generating all $e(l, k)$ satisfying the conditions of Theorem~\textup{\ref{thmMainTheorem}} with $l, k \le 6$, we found that all 0-APN monomials of the form $e(l,k)$ were covered by our theorem in dimensions $2 \le n \le 100$.

\end{remark}

Another consideration that we should take into account is the size of the image set of $x^{e(l,k)}$. It is known that any APN monomial $x^d$ over $\F_{2^n}$ is a bijection if $n$ is odd, and is 3-to-1 on $\F_{2^n}^*$ if $n$ is even. Furthermore, we know that this occurs if and only if $\gcd(d,2^n-1) = 1$ and $\gcd(d,2^n-1) = 3$, respectively. Thus, exponents of the form $e(l,k)$ that do not satisfy this condition can be discarded when searching for new APN monomials. This is the motivation for the following proposition.

\begin{proposition}
  Let $l,k,n$ be natural numbers such that $\gcd(k,n) = 1$. Then
  \begin{equation*}
    \gcd( e(l,k), 2^n-1) = 2^{\gcd(l,n)} - 1.
  \end{equation*}
  In particular, $x^{e(l,k)}$ is a permutation if and only if $\gcd(l,n) = 1$, and it is a $3$-to-$1$ function if and only if $\gcd(l,n) = 2$.
  \label{propImageSet}
  \end{proposition}
  \begin{proof}
    Since $\gcd(A,B) = \gcd(CA,B)$ for $C$ with $\gcd(C,B) = 1$, we have that
    \begin{equation*}
      \gcd( e(l,k),2^n-1 ) = \gcd \left( \frac{2^{lk} - 1}{2^k - 1}, 2^n-1 \right) = \gcd(2^{lk} - 1, 2^n-1)
    \end{equation*}
    due to $\gcd(k,n) = 1$ implying $\gcd(2^k-1,2^n-1) = 1$. Again from $\gcd(k,n) = 1$, we have
    \begin{equation*}
      \gcd(2^{lk} - 1, 2^n-1) = 2^{\gcd(lk,n)}-1 = 2^{\gcd(l,n)} - 1,
    \end{equation*}
    which shows our proposition.
  \end{proof}

\section{Equivalence to known monomial families}

In this section, we study when the exponent $e(l,k)$ can be cyclotomic equivalent to an exponent from one of the known families. Recall that two exponents $e$ and $d$ are cyclotomic equivalent modulo $n$ if either $e$ or $e^{-1}$ is in the cyclotomic coset of $d$ modulo $2^n-1$. For each of the infinite APN families, we treat the two cases separately. In the case of the Dobbertin family, a characterization using our current techniques is too cumbersome due to the large number of degenerate cases that need to be treated before applying Lemma~\ref{lemUniqueExpansion}, and so we supply computational data instead showing that $e(l,k)$ can not be equivalent to the Dobbertin inverse for $n \le 200$ except for $n = 5$ and $n = 10$. We refer the reader to the discussion at the end of Section \ref{secIntroduction} for more details about why we have made this decision.

Below, we shall have arguments dealing with the set (or rather, multiset, since we allow for potential repetitions of elements) of exponents in a sum of powers of $2$ and  we make the convention that if the set contains two ``copies'' of the same element $j$, then the set is {\em compressed} by replacing the two copies of $j$ by $j+1$. For instance, the expression $2^a + 2^b + 2^c$, for some natural numbers $a,b,c$ corresponds to the set of exponents $\{ a,b,c\}$, and if $a = b$ so that $2^a + 2^b + 2^c = 2^{a+1} + 2^c$, then the set compresses to $\{ a+1, c \}$. If $S$ is a set of integers, we will also used the shorthand notation
\begin{equation*}
    S \Mod n = \{ s \Mod n : s \in S \}.
\end{equation*}
Similarly, we will write $A \equiv B \pmod{n}$ if $A \Mod n = B \Mod n$ for two sets $A,B$.

In the sequel, we make use of the following simple observation. It is based on the well-known fact that the binary weight of any two integers in the same cyclotomic coset modulo $2^n-1$ is the same.

\begin{lemma}
  Let $n,M$, $a_1, a_2, \dots, a_M$, $b_1, b_2, \dots, b_M$, and $n$ be natural numbers such that all $a_i$ for $1 \le i \le M$ are distinct modulo $n$, and all $b_i$ for $1 \le i \le M$ are distinct modulo $n$. Suppose that
  \begin{equation}
    \label{eqExponentEquality}
    \sum_{i = 1}^M 2^{a_i} \equiv \sum_{i = 1}^M 2^{b_i} \pmod{2^n-1}.
  \end{equation}
  Then
  \begin{equation*}
    \{ a_i \Mod{n} : 1 \le i \le M \} = \{ b_i \Mod{n} : 1 \le i \le M \}.
  \end{equation*}

  \begin{proof}
    Suppose that \eqref{eqExponentEquality} holds and consider the left-hand side. Since $2^{a_i} \equiv 2^{a_i \Mod n} \pmod{2^n-1}$, we can assume that $a_i < n$ for all $i$. Since by assumption all $a_i$ are distinct modulo $n$, the weight of the sum on the left-hand side will remain unchanged after this modulation, and we will once again have $M$ terms. Similarly, we can modulate the sum on the right-hand side, and thus assume that $b_i < n$ for all $i$. Since all the powers of $2$ on the left-hand side are distinct, their sum cannot be greater than $2^n-1$; the same is true for the right-hand side, and so the assumption that the two sums are congruent in fact implies that they are equal. The claim then follows by the uniqueness of the binary expansion.
  \end{proof}
  \label{lemUniqueExpansion}
\end{lemma}

In many of the following proofs, we will use the fact that we know the algebraic weight of an exponent $d$ from one of the known families, and we would like to select a value of $l$ such that $\wt(e(l,k) \Mod (2^n-1)) = \wt(d)$. Following Theorem \ref{thmMainTheorem}, we will focus on the cases when $\gcd(n,k) = 1$ and $\gcd(n,k) = 2$. To begin with, we can observe that if $l < n$ and $\gcd(n,k) = 1$, then $\wt(e(l,k)) = l$ since all of the exponents $2^{jk}$ in $e(l,k) = \sum_{j} 2^{jk}$ are distinct modulo $n$.

\begin{observation}
  \label{obsCoprime}
  Let $n,l,k$ be natural numbers such that $l < k$ and $\gcd(l,k) = 1$. Then $\wt(e(l,k) \Mod (2^n-1)) = l$.
\end{observation}

The situation when $\gcd(k,n) = 2$ is slightly more complicated, and it is addressed by the following Lemma \ref{lem:gcd2}. Note that in the first few cases we only characterize the weight of $e(l,k)$ modulo $2^n-1$ since this is what we need in the subsequent characterizations. However, in the last case, we show that $e(3n/2 + m, 2t)$ not only has the same weight, but is in fact congruent to $e(m,2t)$ modulo $2^n-1$. This shows that there is no need to consider values of $l$ in $e(l,2t)$ greater than $3n/2$ up to equivalence.

Before proving Lemma \ref{lem:gcd2}, we first prove the following auxiliary result. It is useful on its own (by reducing the number of values of $k$ that have to be considered up to equivalence), and is also used in the proof of Lemma \ref{lem:gcd2}. 

\begin{lemma}
\label{lem2}
The following are true:
\begin{enumerate}
\item[$(i)$]  Let $l,k,m$ be natural numbers with $n = 2m$. Then $e(l,m-k)$ and $e(l,m+k)$ are cyclotomic equivalent modulo $2^n-1$. 
\item[$(ii)$]
Let $l,k,m$ be natural numbers with $n = 2m+1$. Then $e(l,m-k+1)$ and $e(l,m+k)$ are cyclotomic equivalent modulo $2^n-1$. 
\end{enumerate}
  \begin{proof}
   We show $(i)$ first. Let $X = lk + lm + m - k$. We claim that $2^X e(l,m-k) \equiv e(l,m+k) \pmod{2^n-1}$. Recall that we can write
    \begin{equation*}
      e(l,K) = \frac{2^{lK} - 1}{2^K - 1}.
    \end{equation*}
    We now multiply the above for $K = m-k$ by $2^X$ with the aforementioned~$X$. We use the fact that $2m = n \equiv 0 \pmod{n}$, and obtain
    \begin{equation*}
      \begin{split}
	2^X \frac{2^{l(m-k)}-1}{2^{m-k}-1} \equiv & \frac{2^{lm - lk + lk + lm + m - k} - 2^{lk + lm + m -k}}{2^{m-k}-1}  \equiv \frac{2^{m - k} - 2^{lk + lm + m - k}}{2^{m-k}-2^{2m}} \\
	\equiv & \frac{2^{m-k}(1 - 2^{l(k+m)})}{2^{m-k}(1 - 2^{m+k})} \equiv \frac{1 - 2^{l(k+m)}}{1 - 2^{m+k}} \equiv \frac{2^{l(k+m)}-1}{2^{m+k}-1} \pmod{2^n-1}.
      \end{split}
    \end{equation*}
    The claim $(ii)$ follows similarly, by taking $X=l(m+k)+m-k+1$.
  \end{proof}
  \label{lemHalfDimension}
\end{lemma}

\begin{lemma}
\label{lem:gcd2}
Let $\gcd(n, 2t) = 2$, the following statements are true$:$
\begin{enumerate}
    \item[$1.$] $\wt\left(e\left(m, 2t \right) \right) = m$ for any $0 < m < \frac{n}{2};$
    \item[$2.$] $\wt\left(e\left(\frac{n}{2} + m, 2t \right) \right) = \frac{n}{2}$ for any $0 < m < \frac{n}{2};$
    \item[$3.$] $\wt\left(e\left(n + m, 2t \right) \right) = \frac{n}{2} + m$ for any $0 < m < \frac{n}{2};$
    \item[$4.$] $e\left(\frac{3n}{2} + m, 2t \right) \equiv e\left(m, 2t \right) \pmod{2^n - 1}$ for any $0 < m < \frac{n}{2}$.
\end{enumerate}
\end{lemma}
\begin{proof}
  \begin{enumerate}
    \item[$1.$] The first claim is straightforward.
    \item[$2.$] Since $\gcd(n, t) = 1$ then $e\left(\frac{n}{2} + m, 2t \right)$ and $e\left(\frac{n}{2} + m, 2 \right)$ have the same weight, so it suffices to find the weight of $e\left(\frac{n}{2} + m, 2 \right)$. Assuming $0 < m < \frac{n}{2}$, we notice
    \begin{equation*}
        e\left(\frac{n}{2} + m, 2 \right) = \sum_{j = 0}^{\frac{n}{2} - 1} 2^{2j} +   \sum_{j = 0}^{m - 1} 2^{2\left(\frac{n}{2} + j\right)} .
    \end{equation*}
    We note that
    \begin{equation*}
        2^{2\left(\frac{n}{2} + j\right)} \equiv 2^{2j} \pmod{2^n - 1},
    \end{equation*}
    and so
    \begin{align*}
        e\left(\frac{n}{2} + m, 2 \right) &\equiv \left(\sum_{j = 0}^{\frac{n}{2} - 1} 2^{2j}   +  \sum_{j = 0}^{m - 1} 2^{2j} \right) \pmod{2^n - 1} \\
        &\equiv \left( \sum_{j = 0}^{m - 1} 2^{2j + 1}+ \sum_{j = m}^{\frac{n}{2} - 1} 2^{2j} \right) \pmod{2^n - 1}
    \end{align*}
    has weight $m + \left(\frac{n}{2} - m\right)  = \frac{n}{2}$.

    \item[$3.$] Assuming $0 < m < \frac{n}{2}$, we note that
    \begin{equation*}
        e(n + m, 2) =  \sum_{j = 0}^{n - 1} 2^{2j}  + \sum_{j = 0}^{m - 1} 2^{2\left(n + j\right)},
    \end{equation*}
    and that
    \begin{equation*}
        \sum_{j = 0}^{n - 1} 2^{2j} =  \sum_{j = 0}^{\frac{n}{2} - 1} 2^{2j} + \sum_{j = 0}^{\frac{n}{2} - 1} 2^{2(\frac{n}{2} + j)} \equiv \left( \sum_{j = 0}^{\frac{n}{2} - 1} 2^{2j + 1} \right) \pmod{2^n - 1},
    \end{equation*}
    so that
    \begin{equation*}
        e(n + m, 2) \equiv \left( \sum_{j = 0}^{\frac{n}{2} - 1} 2^{2j + 1} + \sum_{j = 0}^{m - 1} 2^{2j} \right) \pmod{2^n - 1}.
    \end{equation*}
    It follows that $e(n + m, 2)$ has weight $\frac{n}{2} + m$. The general case of $t>1$ is treated in the following way. Observe that
    \[
    e\left(n + m, 2t\right)= \sum_{j = 0}^{n - 1} 2^{2tj}  + \sum_{j = 0}^{m - 1} 2^{2t (n+j)} \equiv \sum_{j = 0}^{n - 1} 2^{2tj}  + \sum_{j = 0}^{m - 1} 2^{2t j}\pmod{2^n - 1}.
    \]
    To prove our claim, we will argue that the first sum compresses to  precisely $n/2$ terms, all of which have {\em odd} exponents. For that, we will show that given $0\leq j_1<n/2$, there is a unique $n/2\leq j_2<n$ ($j_2$ cannot be in the interval $[0,n/2)$, as we will see below) such that $ 2^{2tj_1} \equiv  2^{2tj_2}  \pmod{2^n - 1}$ (the situation is similar if we start with $n/2\leq j_1<n$). Via Lemma~\ref{lem2}, we know that one can take $2t<n/2$, so we will assume that from here on in our argument. Since $\gcd(n,2t)=2$, then $\gcd(n,t)=1$. First, given $0\leq j_1<n/2$, we can take $j_2=n/2+j_1$. Then $2^{2tj_2} - 2^{2tj_1} =2^{2tj_1}\left(2^{2t(j_2-j_1)}-1 \right)\equiv 0\pmod{2^n - 1}$, since $n\,|\,2t(j_2-j_1)=nt$. The existence is shown, and next, we show uniqueness (via a Dirichlet principle type argument). If there exist two values $j_2<j_3$, say, such that  for a given $0\leq j_1<n/2$, we have  $2^{2tj_1} \equiv  2^{2tj_2}  \equiv  2^{2tj_3}  \pmod{2^n - 1}$, then $n\,|\, 2t(j_2-j_1)$ and $n\,|\, 2t(j_3-j_1)$ and consequently, $n\,|\, 2t(j_3-j_2)$. However,  because $\gcd(n,t)=1$, it is not possible that both $j_2,j_3$ belong to the interval $[n/2,n)$, since then $j_3-j_2<n/2$ and $n/2$ cannot divide their difference. In the same way, neither $j_2$, nor $j_3$ can belong to the interval $[0,n/2)$ since then $n/2$ could not divide the difference $j_2-j_1$, or, $j_3-j_1$.

    \item[$4.$] Assuming $0 < m < \frac{n}{2}$, we note that
    \begin{equation*}
        e\left(\frac{3n}{2} + m, 2t\right) = \sum_{j = 0}^{\frac{3n}{2} - 1} 2^{2jt}  +  \sum_{j = 0}^{m - 1} 2^{2\left(\frac{3n}{2} + jt\right)}.
    \end{equation*}
    Recalling $n = 2k$, we see that
    \begin{equation*}
        \sum_{j = 0}^{\frac{3n}{2} - 1} 2^{2jt} = \frac{2^{3nt} - 1}{2^{2t} - 1} = \frac{2^{6kt} - 1}{2^{2t} - 1},
    \end{equation*}
    with
    \begin{equation*}
        2^{6kt} - 1 = (2^{2k} - 1)\left(2^{2k(3t - 1)} + 2^{2k(3t - 2)} + \cdots + 2^{2k} + 1\right).
    \end{equation*}
    Notice that 
    \begin{equation*}
        2^{2k(3t - 1)} + \cdots + 2^{2k} + 1 = 4^{k(3t - 1)} + \cdots + 4^{k} + 1 \equiv 3t \equiv 0 \pmod{3},
    \end{equation*}
    so that
    \begin{equation*}
        \frac{2^{6kt} - 1}{2^{2t} - 1} = \frac{3q\left(2^{2k} - 1\right)}{{2^{2t} - 1}},
    \end{equation*}
    for some $q \in \mathbb{N}$. We note that $\gcd(2^{2t} - 1, 2^{2k} - 1) = 3$, and so
    \begin{equation*}
        \frac{3q(2^{2k} - 1)}{2^{2t} - 1} = q'\left(2^{2k} - 1\right) \equiv 0 \pmod{2^{2k} - 1}
    \end{equation*}
    for some $q' \in \mathbb{N}$. It follows that
    \begin{equation*}
        e\left(\frac{3n}{2} + m, 2t\right) \equiv \sum_{j = 0}^{m - 1} 2^{2jt} \pmod{2^n - 1}.
    \end{equation*}
\end{enumerate}
The lemma is shown.
\end{proof}

We also frequently make use of the following observation which follows from the fact that $1 + 2 + \cdots +2^{n-1} \equiv 0 \pmod{2^n-1}$ (see also \cite{CLS11}).

\begin{observation}
  \label{obsBinaryComplement}
  The binary decomposition of $-a$ modulo $2^n-1$ is the complement of that of $a$. 
\end{observation}

For example, $3$ can be written as $(000011)$, i.e. $2^0 + 2^1$ in binary, and $-3 \equiv 60 \pmod{63}$ has the binary expansion $(111100)$, i.e. $2^2 + 2^3 + 2^4 + 2^5$.

\subsection{Gold and Inverse case}

We can see that representatives from some of the known infinite families of APN monomials can be expressed in the form $e(l,k)$. This can be observed quite easily using the formula for the sum of a geometric progression. In particular, the Gold functions $x^{2^k + 1}$ can clearly be expressed as $e(2,k)$. The inverse function can be written as $\displaystyle e(n-1,1) = \sum_{i = 0}^{n-2} 2^i = 2^{n-1} - 1$. We have also observed that in some cases, e.g. for $l = (n-1)/2$ and $k = 2$, or for $l = (n-1)/2 + 1$ and $k = 1$, $e(l,k)$ is equivalent to a Gold function, which is not surprising since the inverse of the Gold function $x^{2^r+1}$ over $\F_{2^n}$, $n$ odd, $\gcd(n,r)=1$, is given by $e(2r, \frac{n+1}{2})=\sum_{i=0}^{\frac{n-1}{2}} 2^{2ir}$. 

\subsection{Welch case}

We note that the Welch exponent is only defined for odd dimensions $n$. Since we assume $\gcd(k,n) \le 2$ for $e(l,k)$ in Theorem \ref{thmMainTheorem} and the following remark, this leaves us with $\gcd(k,n) = 1$ as the only possibility. By Observation \ref{obsCoprime} we know that $\wt(e(l,k) \Mod (2^n-1)) = l$, and since the weight of the Welch exponent $2^t + 2 + 1$ is $3$, it is enough to consider $e(3,k)$.

\begin{theorem}
  Let $n$ be a natural number. Let $W = 2^t + 2 + 1$ be the Welch exponent, where $t > 1$ is some natural number. Suppose $t > 2$. Then $W$ and $e(3,i)$ never lie in the same cyclotomic coset modulo $2^n-1$ for any $0 < i < n$ with $n = 2t + 1$.
  \label{propWelch}
\end{theorem}
\begin{proof}
Suppose that there are natural numbers $a,i < n$ such that
\begin{equation}
  2^a e(3,i) = 2^a ( 2^{2i} + 2^i + 1) \equiv 2^t + 2 + 1 \pmod{2^n-1}.
  \label{eqWelchEquivalence}
\end{equation}
First, we argue that the exponents on the right-hand side are pairwise distinct modulo $n$. Clearly, we cannot have $1 \equiv 0$, or $t \equiv 0$, or $t \equiv 1$ by the hypothesis. The exponents on the left-hand side of \eqref{eqWelchEquivalence} must thus also be distinct modulo $n$ in order for the congruence to hold. By Lemma~\ref{lemUniqueExpansion}, we now have
\begin{equation*}
  \{ 2i+a, i+a, a \} \Mod n = \{ t, 1, 0 \}.
\end{equation*}
We consider several cases depending on which of the three exponents $2i+a$, $i+a$ and $a$ need to be reduced modulo $2^n-1$.

\noindent
\textbf{Case 1:} If $2i+a < n$, then $\{ 2i+a, i+a, a \} = \{ t, 1, 0 \}$. Thus, $a=0,i=1,t=2$.

\noindent
\textbf{Case 2:} Let $i + a < n$ and $2i + a \ge n$. Write $k = 2i + a - n$. We thus have $\{ k, a+i, a \} \Mod n = \{ t, 1, 0 \}$. We cannot have $a + i = 0$ since this implies $a = i = 0$ and leads to a contradiction. We consider two sub-cases:
\begin{enumerate}
  \item if $k = 0$, then we have $a + 2i = n$. We then have $\{ a+i, a \} \Mod n = \{ t, 1 \}$.
  
  If $a + i = t$ and $a = 1$, then from $k = 2i + a - n = 0$, we get $2i + 1 - 2t - 1 = 0$, i.e. $i = t$. Then $a + i = t$ implies $a \equiv 0$, which contradicts $a = 1$.
  
  If $a + i = 1$ and $t = a$, then $k = 2i + a - 2t - 1 = 0$ implies $i = 2t$. Then $a + i = 1$ implies $3t = 1$, which leads to $t = 2$, i.e. $n = 5$.
  
  \item if $k\neq 0$, then $a = 0$, so we have $\{ k, a + i \} = \{ k, i \} = \{ t, 1 \}$. 
  
   We have $k = 2i - 2t - 1$, i.e. $2t + k = 2i - 1$. If $k = t$ and $i = 1$, this means that $3t = 1$ as in the previous sub-case. If $k = 1$ and $i = t$, then we get $1 = -1$.
\end{enumerate}

\noindent\textbf{Case 3:} If $i + a \ge n$, then let $k = i + a - n$. We thus have $\{ k + i, k, a \} = \{ t, 1, 0 \}$. Surely, only $k,a$ could be $0$.  
Once again, we split into sub-cases:
\begin{enumerate}
  \item if $k = 0$, then we have $\{ i, a \} = \{ t, 1 \}$. We have two possibilities:
    \begin{enumerate}
      \item if $i = 1$ and $a = t$, then $0 = k = a + i -n = t + 1 - 2t - 1 = -t$, so $t = 0$ and thus $n = 1$;
      \item if $i = t$ and $a = 1$, then we get $0 = k = t - i $, so that $i = t$. But then $k = a + i - n = a + t -2t - 1 = 1 + t -2t - 1 = -t$, and so once again $t = 0$.
    \end{enumerate}
  \item if $a = 0$, then we have $\{ k + i, k \} = \{ t, 1 \}$. We consider two sub-cases:
    \begin{enumerate}
      \item if $k + i = t$ and $k = 1$, then we have $k = i - n = i - 2t - 1 = 1$, and so $i = 2t + 2 = n + 1$, which contradicts the choice of $i$;
      \item if $k + i = 1$ and $k = t$, then $i = 1 - k = 1 - t$, and since $i \ge 0$, i.e. $1 - t \ge 0$, we have $t \le 1$, i.e. $n \le 3$.
    \end{enumerate}
\end{enumerate}
The claim is shown.
\end{proof}

We now concentrate on the inverse Welch exponent. 
\begin{theorem}
  Let $W = 2^t + 2 + 1$ be the Welch exponent for some natural number $t > 2$. Then $W^{-1} \pmod{2^n-1}$ is never congruent to $e(l,k)$ for any $l < n$ and any $k$ over $\F_{2^n}$ with $n = 2t+1$.
  
  \label{propWelch_inv}
\end{theorem}
\begin{proof}
We assume that for $n>5$ and so, $t>2$,  there are some positive integers $a<n,k<n,1\leq l< n$  such that
\begin{equation}
\label{welch_inv}
e(l,k) (2^t+2+1)\equiv 2^a \pmod{2^n-1}.
\end{equation}
If $k=1$, we get the congruence
\[
(2^l-1)(2^t+2+1)=2^{l+t}+2^{l+1}+2^l-2^t-2-1\equiv 2^a\pmod{2^n-1}.
\]
The left-hand side of the above congruence can be written as
\begin{equation*}
    S = (2^{l+t} - 2^t) + 2^{l+1} + (2^l - 2^2) + 1,
\end{equation*}
and so we have
\begin{equation}
\label{welch_inv2}
S=2^{l+t-1}+2^{l+t-2}+\cdots+2^t+2^{l+1}+2^{l-1}+\cdots+2^2+1\equiv 2^a\pmod {2^n-1}.
\end{equation}
If $l+1<t$ (so, $l+t-1<n$), we get a contradiction by uniqueness of the binary expansion.   

If $l+1=t, t+1$, then $S=2^{l+t}+2^{l-1}+\cdots+2^2+1$, respectively, $S=2^{l+t}+2^l+2^{l-1}+\cdots+2^2+1$. If $l+1=t+2$, then (note that  $l+t=n$, now) $S =2^{l+t}+2^{t+1}+2^{t+1}+2^{l-2}+\cdots +2^2+1\equiv 2^{t+2}+2^{t-1}+\cdots +2^2+2\pmod {2^n-1}$. None of these values of $S$   modulo $2^n-1$ can have Hamming weight~1.

If $l+1\geq t+3$, arranging the sums to see how the cascading (``merger'' of powers of $2$) works, we obtain (writing $l+1=t+s, 3\leq s$, $t + 1 \leq s$, the upper bound comes from $l<n$)
\[
\begin{array}{llll}
S&=2^{l+t-1}+\cdots &+2^{t+s}+2^{t+s-1}&+2^{t+s-2}+\cdots+2^t\\
& &+2^{l+1} &+2^{l-1}+  \cdots+2^{l+1-s} \\
&\qquad &&+ 2^{t-1}+\cdots+2^2+1\\
&= 2^{l+t}+2^{t+s-1} &+2\left(2^{t+s-2}+\cdots+2^t\right)&+2^{t-1}+\cdots+2^2+1 \\
&=2^{l+t}+2^{t+s}&+2^{t+s-2}+\cdots +2^{t+1}&+2^{t-1}+\cdots+2^2+1\\
&\equiv 2^{l-t-1}+2^{l+1}&+2^{l-1}+\cdots +2^{t+1} &+2^{t-1}+\cdots+2^2+1\!\!\!\!\pmod {2^n-1}, 
\end{array}
\]
where we used above that $l+t\pmod n=l+t-n=l-t-1$, and $t+s=l+1$.
Since $t\geq 1$, $l-t-1$ falls either in the set of indices   $\{0\}$, $\{2,\ldots, t-1\}$, or $\{t+1,\ldots,l-1\}$ and since the gaps between these sets are of length $2$, the cascading compression cannot jump into a different set.
Thus,
the expression $S\pmod{2^n-1}$  cannot have Hamming weight~1.

We now take $1<k<n$. Equation~\eqref{welch_inv} is equivalent to
\begin{equation}
    \label{eqB}
B=\sum_{i=0}^{l-1} 2^{ki+t}+\sum_{i=0}^{l-1} 2^{ki+1}+\sum_{i=0}^{l-1} 2^{ki}\equiv 2^a \pmod{2^n-1}.
\end{equation}
If $k(l-1)+t<n$ and $t\not\equiv 0,1\pmod k$, we get a contradiction by the uniqueness of the binary expansion, since the exponents are sitting in different residue classes modulo~$k$. We next assume that $k(l-1)+t<n$, $t\equiv 0\pmod k$, say $t=ks$, $s\geq 1$. Since $k(l-1)+t<n=2t+1$, that is, $k(l-1)<ks+1$, then $s>l-1$, and so, there is no compression in~\eqref{eqB}, and the claim follows via the uniqueness of the binary decomposition.
The case of $k(l-1)+t<n$, $t\equiv 1\pmod k$ follows similarly.

We now let $k(l-1)+t\geq n$, that is, $k(l-1)\geq t+1$. In the same way as above, if $t\not\equiv 0,1\pmod k$, reducing all exponents modulo $n$, we see that they cannot overlap (hence no compression in the sum) since they belong to different residue classes modulo~$k$. We now take $t\equiv 0\pmod k$ (one treats similarly $t\equiv 1\pmod k$), $t=ks, s\geq 1$ and, since $k(l-1)\geq t+1=ks+1$, then $1\leq s< l-1$ . Thus, recalling that $a \Mod N$ denotes the least positive residue of $a$ modulo $N$,
\allowdisplaybreaks
\begin{align*}
B&\equiv 2^{k(l-1+s) \Mod n}+\cdots+2^{kl \Mod n} +\sum_{i=s}^{l-1} 2^{ki+1\Mod n}\\
&\qquad\qquad + 1+2^k+\cdots+2^{k(s-1) \Mod n} +\sum_{i=0}^{l-1} 2^{ki+1\Mod n}\\
&\equiv  2^{k(l-1+s) \Mod n}+\cdots+2^{kl \Mod n}+\sum_{i=s}^{l-1} 2^{ki+2\Mod n}\\
&\qquad\qquad +2^{k(s-1) (\Mod n)}+\cdots+2^k+1+\sum_{i=0}^{s-1} 2^{ki+1\Mod n},
\end{align*}
and again working modulo $k$, we see that this cannot have Hamming weight~1.

Thus, the result holds.
\end{proof}

To conclude the discussion, we observe that for $t = 1$ the Welch exponent is $2^t + 3 = 5$, and can be expressed as $e(2,2)$, while for $t = 2$, the Welch exponent is $2^t + 3 = 7$. Its inverse, $7^{-1} \equiv 9 \pmod{2^5-1}$, can be expressed as $e(2,3)$. As the above theorems demonstrate, these are the only cases in which the Welch exponent can be expressed as $e(l,k)$.
  
\subsection{Kasami case}

The Kasami exponents on $\F_{2^n}$ are defined as $2^{2t} - 2^t + 1$ with $\gcd(t,n) = 1$. The algebraic degree of the Kasami exponent is known (and easily shown) to be~$t+1$.

If $\gcd(i,n) = 1$, this means that we need to consider $l = t + 1$. Since the Kasami exponents are defined for both even and odd $n$, it is possible that we have $\gcd(i,n) = 2$ in addition to $\gcd(i,n) = 1$. According to Lemma \ref{lem:gcd2}, however, even if $\gcd(i,n) = 2$, it is still sufficient to consider $l = t + 1$ due to $t < n/2$ which we can always assume up to equivalence.

\begin{theorem}
  Let $t,n$ be natural numbers such that $t > 2$, $\gcd(t,n) = 1$ and let $K_t = 2^{2t} - 2^t + 1$ be the Kasami exponent for $t < n/2$. Then $K_t$ is never congruent to $e(t+1,i)$ modulo $2^n-1$ for any $i < n$ such that $\gcd(i,n) = 1$.
  
  \label{propKasami}
\end{theorem}
\begin{proof}
Let us assume that there is some $a < n$ such that 
\begin{equation}
    \label{eqKasamiCongruence}
    2^a \cdot (2^{2t} - 2^t + 1) \equiv e(t+1,i)
\end{equation}
for some $i < n$. Depending on which of the exponents in $\{ 2t + a, t + a, a \}$ need to be modulated, we examine several cases, applying Lemma \ref{lemUniqueExpansion} in each case.

\noindent\textbf{Case 1:} If $2t + a < n$, then all exponents are already less than $n$. The exponent on the left-hand side of \eqref{eqKasamiCongruence} modulo $2^n-1$ is thus simply $2^{2t+a} - 2^{t+a} + 2^a$, which can be expressed as
\begin{equation*}
  \left( \sum_{j = a+t}^{2t+a-1} 2^j \right) + 2^a.
\end{equation*}
The set of exponents of this expression is
\begin{equation*}
  A = \{ a \} \cup \{ a + t, a + t + 1, \dots, 2t + a - 2, 2t + a - 1 \}.
\end{equation*}
The set of exponents of $e(t+1,i) = \sum_{j = 0}^t 2^{ji}$ is
\begin{equation*}
  B = \{ 0, i, 2i, \dots, (t-1)i, ti \} \Mod{n}.
\end{equation*}
Note that all elements in $B$ must be distinct modulo $n$ due to $\gcd(i,n) = 1$, and so we can apply Lemma~\ref{lemUniqueExpansion}. Since $0 \in B$, we must also have $0 \in A$. If $0 = a + t + j$ for some $j > 0$, then we get a contradiction because we would have $a + t < 0$. If $0 = a + t$, then $a = t = 0$, but the Kasami function is only defined for $t > 0$. Thus, we must have $a = 0$ and $t \ne 0$. The set $A$ then becomes
\begin{equation*}
  A = \{ 0 \} \cup \{ t, t+1, \dots, 2t-2, 2t-1 \}.
\end{equation*}

There must be $\alpha, \beta$ in $0 \le \alpha, \beta \le t$ such that $\alpha i \equiv t \pmod{n}$ and $\beta i \equiv t+1 \pmod{n}$. Then either $\alpha - \beta$ or $\beta - \alpha$ is in the range between $0$ and $t$, so either $(\alpha - \beta)i \Mod n$ or $(\beta - \alpha)i \Mod n$ belongs to $A$. But $(\beta - \alpha)i \equiv 1 \notin A$, and so we must have $(\alpha - \beta)i \equiv -1 = n-1 \in A$, which is impossible under $t < n/2$.

\noindent\textbf{Case 2:} If $t + a < n \le 2t + a$, then let $k = 2t + a - n$. We must have $k < t + a$, otherwise $k = 2t + a -n \ge t + a$, hence $t - n \ge 0$, i.e. $t \ge n$. Similarly, we can assume $k < a$ under $t \le n/2$. Then $k < a < t+a$. Using Observation~\ref{obsBinaryComplement}, we have that the Kasami exponent becomes
\begin{equation*}
  \begin{split}
  -2^{t+a} + 2^a + 2^k = -(2^{t+a} - 2^a) + 2^k = 2^k + \sum_{j = 0}^{a-1} 2^j + \sum_{j = t+a}^{n-1} 2^j = \\
  \left( \sum_{j = 0}^{k-1} 2^j \right) + 2^a + \left( \sum_{j = t+a}^{n-1} 2^j \right),
  \end{split}
\end{equation*}
which gives the set of exponents
\begin{equation*}
  A = \{ 0,1,2,\dots,k-1\} \cup \{a\} \cup \{t+a,t+a+1,\dots,n-1\}.
\end{equation*}
The set of exponents of $e(t+1,i)$ is still the same as before, i.e.
\begin{equation*}
  B = \{ 0, i, 2i, \dots, (t-1)i, ti \} \Mod{n}.
\end{equation*}

We must have $\alpha, \beta$ in $0 \le \alpha, \beta \le t$ such that $\alpha i \equiv t+a \pmod{n}$ and $\beta i \equiv a \pmod{n}$. Then either $\alpha - \beta$ or $\beta - \alpha$ is in the range $\{ 0, 1, \ldots, t \}$, and so we must have either $t \in A$ or $-t \equiv n - t \in A$.

If $t \in A$, then we must have $t = a$ since $t \le k-1$ means $t \le 2t + a - n - 1$, i.e. $t + a \ge n + 1$ which contradicts $t + a < n$; and $t \ge t + a$ implies $a = 0$ which contradicts that same assumption. If $a = t$, then the set of exponents becomes
\begin{equation*}
    A = \{ 0, 1, \ldots, k-1 \} \cup \{ t \} \cup \{ 2t, 2t+1, \ldots, n-1 \}.
\end{equation*}
Now, take some $\gamma, \delta$ such that $0 \le \gamma, \delta \le t$ and $\gamma i \equiv 1 \pmod{n}$ and $\delta i \equiv t \pmod{n}$. Then either $t - 1$ or $n + 1 - t$ must be in $A$, and we can verify that this is impossible: if $t - 1 \le k - 1$, then $t \le 3t - n$, i.e. $2t \ge n$ contradicting $t < n/2$; if $t - 1 = t$ or $t - 1 \ge 2t$, then we immediately get a contradiction as well. Similarly, $n + 1 - t \le k - 1$ means $n + 1 - t \le 3t - 1$, i.e. $4t \ge 2n + 2$ contradicting the choice of $t$; $n + 1 - t = t$ leads to $2t = n + 1$, violating that same hypothesis; and $n + 1 - t \ge t + a$ means $n \ge 3t - 1$, but together with $3t \ge n$ and $\gcd(t,n) = 1$ this leads to $3t = n + 1$ so that $k = 3t - n = 1$. Thus $A = \{ 0 \} \cup \{ t \} \cup \{ 2t, 2t + 1, \ldots, 3t - 2 \}$. Then taking $\varepsilon i \equiv 3t - 2 \pmod{n}$, we have either $(3t - 2) - t = 2t - 2 \in A$, or $t - (3t - 2) = n - 2t + 2 \in A$. The former case is clearly impossible except potentially for trivial values of $t$. In the latter case, we must have $n + 2 - 2t \ge 2t$, i.e. $n + 2 = 3t + 1 \ge 4t$, which again implies a contradiction.

If $n - t \in A$, we derive a contradiction in a similar manner.

\noindent\textbf{Case 3:} If $t + a > n$, let $k = t + a - n$. Then
\begin{equation*}
  2^{2t+a} - 2^{t+a} + 2^a \equiv 2^{k+a} - 2^k + 2^a \pmod{2^n-1}.
\end{equation*}
Note that $k < a$ since otherwise we get $t + a - n \ge a$, i.e. $t \ge n$. We need to examine two sub-cases depending on whether $k+a$ needs to be modulated or not.

\noindent\textbf{Case 3-1:} If $k+a < n$, then the above becomes
\begin{equation*}
  2^{k+a} + 2^a - 2^k = \left( \sum_{j = k}^{a-1} 2^j \right) + 2^{k+a},
\end{equation*}
giving the set of exponents
\begin{equation*}
  A = \{ k, k+1, \dots, a-1 \} \cup \{ k + a \}
\end{equation*}
which has to be congruent with
\begin{equation*}
  B = \{ ti, (t-1)i, \dots, 2i, i, 0 \} \Mod n.
\end{equation*}
As before, if $0 = k+j$ for some $j > 0$, then we immediately get a contradiction due to $k,j \ge 0$. So $k = 0$, i.e. $t + a = n$. The above sets become
\begin{equation*}
  A = \{ 0, 1, 2, \dots, a-1, a \}
\end{equation*}
and
\begin{equation*}
  B = \{ ti, (t-1)i, \dots, 2i, i, 0 \} \Mod n.
\end{equation*}
Clearly, we must have $a = t$ in order for equality to hold (by comparing the number of terms), but then $t + a > n$ and $k + a < n$ cannot hold simultaneously.

\noindent\textbf{Case 3-2:} If $k + a \ge n$, then let $q = k + a - n = t + 2(a-n)$. We have $q < k$ since if $q \ge k$, then $k + a - n \ge k$, i.e. $a \ge n$. Similarly, $k < a$ since if $k \ge a$ then $t + a -n \ge a$, i.e. $t \ge n$. Thus we have $q < k < a$. The Kasami exponent is thus 
\begin{equation*}
  2^a - 2^k + 2^q = \left( \sum_{j = k}^{a-1} 2^j \right) + 2^q.
\end{equation*}
The set of exponents is
\begin{equation*}
  A = \{ q \} \cup \{k, k+1, k+2, \dots, a-1 \}.
\end{equation*}
Since $0 \in A$ due to $0 \in B$, the only possibility is $q = 0$, i.e. $k + a = n$. Comparing the sizes of $A$ and $B$, we get $a - 1 - k + 1 + 1 = t + 1$, i.e. $a - k =t$, hence $n = 2t$, which contradicts $\gcd(t,n) = 1$.

This concludes the proof of this case.
\end{proof}

While the inverse of the Kasami power function is known~(see~\cite[Theorem 3.10]{KS14}, as well as~\cite{K20}, for further clarification), it seems complicated to investigate its cyclotomic equivalence to some $e(l,k)$, since the inverse formula depends upon the inverse of the Kasami exponent $2^{2t} - 2^t + 1$ modulo $2^r-1$, where $r$ is the least positive residue of $n$ modulo $6k$, and that is not explicit.  

However, we find a way around that and are able to show the following theorem. 
\begin{theorem}
\label{propKasami_inv}
The inverse of the Kasami exponent $K_t=2^{2t}+2^{t}-1$ is cyclotomic equivalent to $e(l,i)$ over $\F_{2^n}$, $\gcd(n,t)=1$, if and only if $t=1$, or $t=n-1$, and this happens for $(l,i)=e(2, \frac{n+1}{2})$ (hence $K_1^{-1}$ and $K_{n-1}^{-1}$ are in the same cyclotomic class).  
\end{theorem}
\begin{proof}
Suppose that the Kasami exponent $K_t=(2^{2t} - 2^t + 1)$  satisfies $(2^{2t} - 2^t + 1)e(l,i) \equiv 2^a \pmod{2^n-1}$ for some $a,l,i,n$ satisfying the appropriate hypotheses. This means that
\begin{equation}
\label{eq:invKasami}
    \frac{2^{li} - 1}{2^i - 1} (2^{2t} - 2^t + 1) \equiv 2^a \pmod{2^n-1}.
\end{equation}
We first make some interesting observations. If $t=1$, the Kasami function is also a Gold function, and we have already treated that case. If $t=n-1$ (which is coprime to $n$),  we first observe that the Kasami function is then $2^{2t}-2^t+1=2^{2n-2}-2^{n-1}+1\equiv 2^{n-2}-2^{n-1}+1\equiv -2^{n-2}\equiv 3\cdot 2^{n-2}\pmod {2^n-1}$. Thus,  taking $i=2$, $l=\frac{n+1}{2},a=n-2$,
Equation~\eqref{eq:invKasami} becomes 
\[
3\cdot 2^{n-2}  \frac{2^{n+1} - 1}{2^2 - 1}=
2^{n-2}\equiv 2^{n-2}\pmod {2^n-1}.
\]
Thus, the inverse of $K_{n-1}$ is in the cyclotomic class of $e\left(\frac{n+1}{2},2\right)$.

We therefore assume that $1<t<n-1$.
Multiplying both sides of~\eqref{eq:invKasami} by $2^i - 1$  and regrouping the terms on both sides yields
\begin{equation}
\label{eq:Case2_invKasami}
    2^{2t + li} + 2^{li} + 2^t + 2^a \equiv 2^{2t} + 2^{li + t} + 2^{a+i} + 1 \pmod{2^n-1}.
\end{equation}

We will apply Lemma~\ref{lemUniqueExpansion} to the exponents on the left-hand and right-hand side of the above identity. In order to do so, we first need to treat the cases when one or more of the integers in the exponent sets $$A = \{ 2t + li, li, t, a \}$$ and $$B = \{ 2t, li+t, a+i, 0 \}$$ are congruent to each other modulo $n$.

We consider the elements in $B = \{ 2t, li + t, a + i, 0 \}$. If not all of them are distinct modulo $n$, then we must be in one of the following cases.

\noindent\textbf{Case 1:} If $0 \equiv 2t \pmod{n}$, then $t \equiv 0 \pmod{n}$ due to $n$ being odd, which contradicts $2^{2t} - 2^t + 1$ being APN.

\noindent\textbf{Case 2:} If $li + t \equiv 0 \pmod{n}$, then the integers in $A$ and $B$ become (after compression)
\begin{align*}
    A = \{ t+1, -t, a \},\
    B = \{ 2t, a+i, 1 \}.
\end{align*}

We now consider several cases depending on whether the elements in the reduced set $A$ are congruent to each other modulo $n$ or not.

\noindent\textbf{Case 2.1:} If all the elements in $A$ are distinct modulo $n$, then the same must be true for those in $B$. We can now apply Lemma~\ref{lemUniqueExpansion}, which implies that $A \equiv B \pmod{n}$. Thus, $t + 1$ must be congruent to one of the elements in $B$. 
Surely, since $1<t<n-1$, then  $t + 1 \not\equiv 1 \pmod{n}$. 
If $t + 1 \equiv a+i \pmod{n}$, then we must have $-t \equiv 1 \pmod{n}$ (since $-t \equiv 2t \pmod{n}$ immediately implies $t = 0$), i.e. $t \equiv -1 \pmod{n}$.
If $t + 1 \equiv 2t \pmod{n}$, then we immediately get $t \equiv 1 \pmod{n}$.

\noindent\textbf{Case 2.2:} Suppose now that $t + 1 \equiv -t \pmod{n}$. Then the sets $A$ and $B$ become
\begin{align*}
    A = \{ -t, -t, a \} = \{ -t + 1, a \},\
    B = \{ 2t, a+i, 1 \} = \{ -1, a+i, 1 \}.
\end{align*}
Since $A$ consists of at most two distinct powers of $2$ modulo $n$, then at least two of the integers in $B$ must be congruent to each other as well.
\begin{itemize}
    \item
Clearly, $-1 \equiv 1 \pmod{n}$, i.e. $2 \equiv 0 \pmod{n}$ is impossible unless $n \le 2$.
\item
If $a + i \equiv 1 \pmod{n}$, then we have $A = \{ 1 - t, a \}$ and $B = \{ -1, 2 \}$. If the two elements in $B$ are congruent to each other, i.e. $-1 \equiv 2 \pmod{n}$, then we get $3 \mid n$, and so $n \le 3$. Otherwise, we apply Lemma~\ref{lemUniqueExpansion} and equate the two sets modulo $n$. If $1 - t \equiv 2 \pmod{n}$, then we immediately get $t \equiv - 1 \pmod{n}$, i.e. $t = n-1$, which is one of the cases where the Kasami inverse can be equivalent to $e(l,i)$; whereas, if $1 - t \equiv -1 \pmod{n}$ and $a \equiv 2 \pmod{n}$, then we get $t = 2$, and from $2t \equiv -1 \pmod{n}$, we have $5 \equiv 0 \pmod{n}$, implying $n \le 5$.
\item
If $a + i \equiv -1 \pmod{n}$, then we get $A = \{ 1 - t, a \}$ and $B = \{ 0, 1 \}$. The elements in $B$ are clearly not congruent modulo $n$, and so we must have $a = 1$ and $1 - t \equiv 0 \pmod{n}$, i.e. $t = 1$. This leads us to the Gold case once again.
\end{itemize}

\noindent\textbf{Case 2.3:} Suppose now that $t + 1 \equiv a \pmod{n}$. The sets $A$ and $B$ become
\begin{align*}
    A = \{ a+1, -t \},\
    B = \{ 2a-2, a + i, 1 \}.
\end{align*}

Once again, since $A$ contains at most two distinct elements modulo $n$, then the same must be true for $B$.
\begin{itemize}
    \item
If $2a - 2 \equiv a + i \pmod{n}$, then we get $a \equiv i + 2 \pmod{n}$, i.e. $i \equiv a-2 \pmod{n}$. Now, $A = \{ a + 1, -t \}$ and $B = \{ a + i + 1, 1 \}$. If the two integers in $A$ are congruent to each other, then the same must be true for the ones in $B$, so that $a + 1 \equiv -t \pmod{n}$ and $a + i \equiv 0 \pmod{n}$, hence $2a - 2 \equiv 0 \pmod{n}$, i.e. $a \equiv 1 \pmod{n}$ and therefore $t \equiv 0 \pmod{n}$ from $t + 1 \equiv a \pmod{n}$. Otherwise, by Lemma~\ref{lemUniqueExpansion}, either $t \equiv -1 \pmod{n}$, or $a + 1 \equiv 1 \pmod{n}$, which from $t + 1 \equiv a \pmod{n}$ implies the same.

\item 
If $2a - 2 \equiv 1 \pmod{n}$, then together with $t + 1 \equiv a \pmod{n}$, we get $2t + 2 \equiv 2a \pmod{n}$ and so $2t \equiv 1 \pmod{n}$, hence $t = \frac{n+1}{2}$. The two sets are $A = \{ a + 1, -t \}$ and $B = \{ 2, a+i \}$. If $a + 1 \equiv -t \pmod{n}$ and $2 \equiv a+i \pmod{n}$, then $a \equiv -1 -t \equiv 1 + t \pmod{n}$, and so $t \equiv -1 \pmod{n}$. Otherwise, we apply Lemma~\ref{lemUniqueExpansion}. If $a + 1 \equiv 2 \pmod{n}$, i.e. $a \equiv 1 \pmod{n}$, then $2a - 2 \equiv 1 \pmod{n}$ implies $1 \equiv 0 \pmod{n}$. If $a + 1 \equiv a + i \pmod{n}$ and $-t \equiv 2 \pmod{n}$, then $i \equiv 1$ and $t \equiv -2 \pmod{n}$. From $t + 1 \equiv a \pmod{n}$, we can derive $a \equiv -1 \pmod{n}$, and substituting this into $2a - 2 \equiv 1 \pmod{n}$ yields $5 \equiv 0 \pmod{n}$, i.e. $n \le 5$.

\item
If $a + i \equiv 1 \pmod{n}$, then the sets are $A = \{ a + 1, -t \} = \{ a + 1, 1 -a \}$ and $B = \{ 2a - 2, 2 \}$. If $a + 1 \equiv 1 -a \pmod{n}$, then we get $a \equiv 0 \pmod{n}$ and hence $t \equiv -1 \pmod{n}$. If $a + 1 \equiv 2a - 2 \pmod{n}$ and $1-a \equiv 2 \pmod{n}$, then $a \equiv - 1 \pmod{n}$, and so $0 \equiv 4 \pmod{n}$. If $a + 1 \equiv 2 \pmod{n}$ and $1 - a \equiv 2a - 2 \pmod{n}$, then we get $a = 1$ and so $t = a - 1 = 0$, which cannot possibly happen.
\end{itemize}

\noindent\textbf{Case 2.4:} If $-t \equiv a \pmod{n}$, then the sets become
\begin{align*}
    A = \{ t + 1, 1 - t \} = \{ t + 1, a + 1 \},\
    B = \{ 2t, i-t, 1 \}.
\end{align*}
At least two of the elements in $B$ must be congruent to each other.

\begin{itemize}
    \item
If $2t \equiv i-t \pmod{n}$, then $B$ becomes $B = \{ 2t+1, 1 \}$. Clearly, we cannot have $2t + 1 \equiv 1 \pmod{n}$, and so an application of Lemma~\ref{lemUniqueExpansion} yields $t + 1 \equiv 2t + 1 \pmod{n}$ or $t + 1 \equiv 1 \pmod{n}$; in both cases, we get $t \equiv 0 \pmod{n}$, which cannot happen.

\item
If $2t \equiv 1 \pmod{n}$, then $B$ becomes $B = \{ 2, i-t \}$. If the elements in $A$ are not distinct, we immediately get $t \equiv -t \pmod{n}$. Otherwise, $t + 1 \equiv 2 \pmod{n}$ yields $t \equiv 1 \pmod{n}$, while if $t + 1 \equiv i-t \pmod{n}$ and $a + 1 \equiv 2 \pmod{n}$, then $a \equiv 1 \pmod{n}$ and $-t \equiv a \pmod{n}$ imply $t \equiv -1 \pmod{n}$.

\item
Finally, if $i-t \equiv 1 \pmod{n}$, then we get $A = \{ t + 1, a + 1 \}$ and $B = \{ 2t, 2 \}$. If $2t \equiv 2 \pmod{n}$, then $t = 1$; otherwise, by Lemma~\ref{lemUniqueExpansion} we must have either $t + 1 \equiv 2t \pmod{n}$, i.e. $t = 1$; or $t + 1 \equiv 2 \pmod{n}$, i.e. $t = 1$. All of these imply the Gold case which we have already handled.
\end{itemize}

\noindent\textbf{Case 3:} If $0 \equiv a + i \pmod{n}$, i.e. $a \equiv -i \pmod{n}$, then we have
\begin{align*}
    A = \{ 2t - la, t, -la, a \},\
    B = \{ 2t, t - la, 0, 0 \} = \{ 2t, t - la, 1 \}.
\end{align*}

Once again, we consider sub-cases depending on which of the elements in $A$ coincide modulo $n$.

\noindent\textbf{Case 3.1:} If $2t - la \equiv t \pmod{n}$, i.e. $t \equiv la \pmod{n}$, then
\begin{align*}
    A = \{ t+1, -t, a \},\
    B = \{ 2t, 0, 1 \}.
\end{align*}

If the elements in $A$, respectively $B$, are all distinct modulo $n$, then by Lemma~\ref{lemUniqueExpansion} we must have $a \equiv 0 \pmod{n}$, but then $2t \equiv t \pmod{n}$ implying $t \equiv 0 \pmod{n}$.

If some of the elements of $B$ coincide modulo $n$, then we must have $2t \equiv 1 \pmod{n}$, and two of the elements in $A$ must collide two. If $t + 1 \equiv -t \pmod{n}$, then we obtain $2t \equiv -1 \pmod{n}$ in addition to $2t \equiv 1 \pmod{n}$, and so $1 \equiv -1 \pmod{n}$. If $t+1 \equiv a \pmod{n}$, then we get $A = \{ a+1, t+1 \}$ and $B = \{ 0, 2 \}$, and $A \equiv B \pmod{n}$ implies trivial values of $t$ in all cases. If $a \equiv -t \pmod{n}$, then $A = \{ t+1, 1-t \}$ and $B = \{ 0, 2 \}$, implying $t \equiv 0 \pmod{n}$, $t \equiv -1 \pmod{n}$, or $t \equiv 1 \pmod{n}$.

\noindent\textbf{Case 3.2:} If $2t - la \equiv -la\pmod{n}$, then $2t \equiv 0\pmod{n}$ giving a contradiction.

\noindent\textbf{Case 3.3:} If $2t - la \equiv a \pmod{n}$, then we have
\begin{align*}
    A = \{ a + 1, t, a-2t \},\
    B = \{ 2t, a-t, 1 \}.
\end{align*}

\noindent\textbf{Case 3.3.A:} If $a + 1 \equiv t\pmod{n}$, then $A = \{ t +1, -1-t \}$ and $B = \{ 2t, - 1, 1 \}$. At least two of the elements of $B$ must coincide modulo $n$ in order for this to hold. 

\begin{itemize}
    \item
If $2t \equiv -1 \pmod{n}$, then $A = \{ t+1, -1 -t \}$ and $B = \{ 0,1 \}$, which leads to either $0 \equiv 1 \pmod{n}$ (if the elements in $B$ coincide) or to $t equiv -1 \pmod{n}$ by Lemma~\eqref{lemUniqueExpansion} (if they are distinct modulo $n$).

\item
If $2t \equiv 1\pmod{n}$, then $A = \{ -1,2 \}$ and $B = \{ t+1, -1 -t \}$. Unless $-1 \equiv 2 \pmod{n}$, we can apply Lemma \ref{lemUniqueExpansion} to obtain either $t+1 \equiv 2 \pmod{n}$, i.e. $t \equiv 1 \pmod{n}$, or $t+1 \equiv -1 \pmod{n}$, i.e. $t \equiv -2 \pmod{n}$. Combined with $2t \equiv 1 \pmod{n}$, we obtain $5 \equiv 0 \pmod{n}$, so that $n \le 5$.
\end{itemize}

\noindent\textbf{Case 3.3.B:} If $a + 1 \equiv a - 2t\pmod{n}$, then $2t \equiv -1\pmod{n}$. The sets become $A = \{ a+2, t \}$ and $B = \{ 2t, a-t, 1 \} = \{ -1, a-t, 1 \}$.

\begin{itemize}
    \item If $1 \equiv a - t \pmod{n}$, then we must equate the sets $A = \{ a+2, t \}$ and $\{ -1, 2 \}$, and this leads to either $t \equiv -1 \pmod{n}$, or to $5 \equiv 0 \pmod{n}$.

    \item If $-1 \equiv a - t \pmod{n}$, then we have $A = \{ a+2, t \}$ and $B = \{ 0,1 \}$ which clearly leads to trivial cases or contradictions as before.
\end{itemize}

\noindent\textbf{Case 3.3.C:} If $t \equiv a -2t\pmod{n}$, then $3t \equiv a\pmod{n}$ and $2t \equiv a -t\pmod{n}$. We have $A = \{ a + 1, t + 1 \}$ and $B = \{ a-t+1, 1 \}$.
\begin{itemize}
    \item If $a + 1 \equiv 1 \pmod{n}$, then $a \equiv 0 \pmod{n}$, and so $t \equiv a - 2t \pmod{n}$ implies $t \equiv 0 \pmod{n}$.
    \item If $t + 1 \equiv 1 \pmod{n}$, we obtain the same contradiction as above.
    \item If $a+1 \equiv t+1 \pmod{n}$ and $a-t+1 \equiv 1 \pmod{n}$, then $a \equiv t \pmod{n}$ and $t \equiv a-2t \pmod{n}$ implies $t \equiv 0 \pmod{n}$.
\end{itemize}

\noindent\textbf{Case 3.3.D:} If all elements in $A$, respectively,  $B$ are distinct modulo $n$, then $A \equiv B \pmod{n}$. We consider several sub-cases depending on which element of $A$ is congruent to $1$.

\begin{itemize}
    \item If $a + 1 \equiv 1 \pmod{n}$, then $a \equiv 0 \pmod{n}$ and we have $A = \{ 1,t,-2t \}$ and $B = \{ 2t, -t, 1 \}$. Clearly, we get $t \equiv 0 \pmod{n}$ in all cases.

    \item If $t \equiv 1 \pmod{n}$, we are already in a trivial case.

    \item If $a-2t \equiv 1 \pmod{n}$, then $A = \{ a+1, t, 1 \}$ and $B = \{ 2t, a-t, 1 \}$. Unless $t \equiv 2t \pmod{n}$, we must have $t \equiv a-t \pmod{n}$ and $a+1 \equiv 2t \pmod{n}$, which then implies $1 \equiv 0 \pmod{n}$.
\end{itemize}

\noindent\textbf{Case 3.4:} If $t \equiv a\pmod{n}$, then we have
\begin{align*}
    A = \{ 2a - la, a + 1, -la \},\
    B = \{ 2a, a - la, 1 \}.
\end{align*}

We consider subcases depending on the structure of $A$.

\noindent\textbf{Case 3.4.A:} If $2a - la \equiv a + 1\pmod{n}$, then $a \equiv 1 + la\pmod{n}$. The sets now become $A = \{ a + 2, 1-a \}$ and $B = \{ 2a, 2 \}$. If the elements in $A$, respectively,  $B$ coincide modulo $n$, then we get $2a \equiv -1 \equiv 2\pmod{n}$, and so $n \le 3$. Otherwise, by Lemma~\ref{lemUniqueExpansion} we have either $a + 2 \equiv 2\pmod{n}$ implying $a = t = 0$, or $a + 2 \equiv 2a\pmod{n}$ and $1-a \equiv 2\pmod{n}$, i.e. $a \equiv -1\pmod{n}$ and $a \equiv 2\pmod{n}$ implying $n \le 3$.

\noindent\textbf{Case 3.4.B:} If $2a - la \equiv -la\pmod{n}$, then $a = t = 0$.

\noindent\textbf{Case 3.4.C:} If $a + 1 \equiv - la\pmod{n}$, then $A = \{ 3a + 1, a + 2 \}$ and $B = \{ 2a + 1, 2a, 1 \}$. Since two terms of $B$ must be congruent modulo $n$, we get $2a \equiv 1\pmod{n}$, so that $B$ becomes $B = \{ 2a + 1, 2 \}$. If the elements in the sets are pairwise congruent modulo $n$, then $A$ collapses to $\{ a + 3 \}$ while $B$ collapses to $\{3 \}$, implying $a \equiv 0 \equiv t \pmod{n}$, which cannot happen. Thus, we must have $a + 2 \equiv 2a + 1\pmod{n}$, so that $t \equiv a \equiv 1\pmod{n}$.

\noindent\textbf{Case 3.4.D:} If all elements of $A$, respectively,  $B$ are distinct modulo $n$, then we apply Lemma~\ref{lemUniqueExpansion}.
If $2a - la \equiv 2a\pmod{n}$, then $lt \equiv 0\pmod{n}$ so $l \equiv  0\pmod{n}$ which cannot happen.
If $2a - la \equiv a -la\pmod{n}$, then $a = 0 = t$.

If $2a - la \equiv 1$ and $a + 1 \equiv a -la\pmod{n}$, then we get $2a + 1 \equiv 1\pmod{n}$ and so $a = t = 0$.

\noindent\textbf{Case 3.5:} If $t \equiv -la\pmod{n}$, then we have
\begin{align*}
    A = \{ 3t, t+1, a \},\
    B = \{ 2t + 1, 1 \}.
\end{align*}

\begin{itemize}
    \item
If $3t \equiv t + 1\pmod{n}$, then $2t \equiv 1\pmod{n}$, so $A = \{ t + 2, a \}$ and $B = \{ 1,2 \}$. This leads to either $n \le 1$, $t = 1$, or $t = 0$.

\item
If $3t \equiv a$, then $A = \{ a+1, t + 1 \}$ and $B = \{ 2t + 1, 1 \}$. This leads to $t = 0$ in all cases.

\item
If $t + 1 \equiv a\pmod{n}$, then $A = \{ 3t, t + 2 \}$ and $B = \{ t + a, 1 \}$. If $3t \equiv t+2 \pmod{n}$, then we directly get $t \equiv 1 \pmod{n}$. Otherwise, we apply Lemma \ref{lemUniqueExpansion} to get either the trivial case $t \equiv -1 \pmod{n}$ if $t+2 \equiv 1 \pmod{n}$, or alternatively $t +2 \equiv t + 2 \pmod{n}$, yielding $a = 2$, and implying $t = 1$ from $t + 1 \equiv a \pmod{n}$.
\end{itemize}

\noindent\textbf{Case 3.6:} If $a \equiv -la\pmod{n}$, then
\begin{align*}
    A = \{ 2t + a, t, a + 1 \},\
    B = \{ 2t, t+a, 1 \}.
\end{align*}

\noindent\textbf{Case 3.6.A:} If $2t + a \equiv t\pmod{n}$, then $t \equiv -a\pmod{n}$, and we have $A = \{ t+1, 1-t \}$ and $B = \{ 2t, 0, 1 \}$. Clearly, $1 \equiv 0\pmod{n}$ and $2t \equiv 0\pmod{n}$ imply trivial cases or a contradiction, so we must have $2t \equiv 1\pmod{n}$ so that $B = \{ 0, 2 \}$. Now, we must either have $0 \equiv 2\pmod{n}$, $t = 1$, or $t \equiv -1\pmod{n}$.

\noindent\textbf{Case 3.6.B:} If $2t + a \equiv a + 1\pmod{n}$, then $2t \equiv 1$, and so $A = \{ a + 2, t \}$ and $B = \{ t + a, 2 \}$. If $a + 2 \equiv t\pmod{n}$ and $a + t \equiv 2\pmod{n}$, then we deduce $2 - t \equiv t - 2\pmod{n}$, i.e. $2t \equiv 4\pmod{n}$, and so $4 \equiv 1\pmod{n}$. If $t \equiv a + t\pmod{n}$, we get $a = 0$, and hence $A = \{ 0, t \}$ and hence $i = 0$ by the assumption $a + i \equiv 0 \pmod{n}$. If $a + 2 \equiv a+t\pmod{n}$ and $t \equiv 2\pmod{n}$, then $2t \equiv 1\pmod{n}$ implies $3 \equiv 0\pmod{n}$.

\noindent\textbf{Case 3.6.C:} If $t \equiv a+1\pmod{n}$, i.e. $a \equiv t-1\pmod{n}$, then $A = \{ t-1, t+1 \}$ and $B = \{ 2t, 2t-1,1 \}$. The only possibility for two elements in $B$ to coincide is when $2t \equiv 1\pmod{n}$, so $B = \{ 2t-1, 2 \}$. Now, if $t-1 \equiv t+1\pmod{n}$, we get $2 \equiv 0\pmod{n}$; if $t-1 \equiv 2\pmod{n}$ and $t + 1 \equiv 2t-1\pmod{n}$, then $t \equiv 3\pmod{n}$ and $t \equiv 2\pmod{n}$; and if $t-1 \equiv 2t-1\pmod{n}$, we immediately get a contradiction.

\noindent\textbf{Case 3.6.D:} If all elements in $A$, respectively,  $B$ are distinct, then by Lemma~\ref{lemUniqueExpansion}, we should have $t \equiv t+a$, so that $a = 0$, but then $a \equiv -i\pmod{n}$ contradicts the exponent being APN. The only remaining two possibilities are $t \equiv 2t\pmod{n}$, implying an immediate contradiction, and $t \equiv 1\pmod{n}$, which is the Gold case.

\noindent\textbf{Case 4:} If $2t \equiv a + i\pmod{n}$, then we have
\begin{align*}
    A = \{ 2t + li, t, li, a \},\
    B = \{ 2t, li + t, 2t, 0 \} = \{ 2t + 1, li + t, 0 \}.
\end{align*}
At least two of the elements of $A$ must now collide.

\noindent\textbf{Case 4.1:} If $2t + li \equiv t\pmod{n}$, then $t \equiv -li\pmod{n}$, and so $A = \{ t + 1, -t, a \}$ and $B = \{ 2t + 1, 1 \}$. Thus, another two of the elements of $A$ must collide.

\begin{itemize}
    \item
If $t + 1 \equiv -t\pmod{n}$ or $t + 1 \equiv a\pmod{n}$, then $A$ contains $t + 2$, while $B$ remains $B = \{ 2t + 1, 1 \}$. Then $t + 2 \equiv 2t + 1\pmod{n}$ and $t+2 \equiv 1\pmod{n}$ both imply $t = 1$, while $2t + 1 \equiv 1\pmod{n}$ gives a contradiction.

\item
If $a \equiv -t\pmod{n}$, then $A = \{ t + 1, 1-t \}$ and $B = \{ 2t + 1, 1 \}$, which leads to $t = 0$ in all cases.
\end{itemize}

\noindent\textbf{Case 4.2:} If $2t + li \equiv li\pmod{n}$, we immediately get $t = 0$.

\noindent\textbf{Case 4.3:} If $2t + li \equiv a\pmod{n}$, then $A = \{ a + 1, t, li \}$ and $B = \{ 2t + 1, li + t, 0 \}$.

\noindent\textbf{Case 4.3.A:} If $2t \equiv -1\pmod{n}$, then $B = \{ 1, li + t \}$. Now, two of the elements in $A = \{ a + 1, t, li \}$ must collide. 

\begin{itemize}
    \item
If $a + 1 \equiv t\pmod{n}$, then $A = \{ t + 1, li \}$ and $B = \{ li + t, 1 \}$. If $t + 1 \equiv li\pmod{n}$ and $li + t \equiv 1\pmod{n}$, then $t \equiv 0\pmod{n}$. Otherwise, we must have $t + 1 \equiv li + t\pmod{n}$ and $li \equiv 1\pmod{n}$, hence $2t + li \equiv a\pmod{n}$ implies $2t + 1 \equiv a\pmod{n}$. But since $2t \equiv -1\pmod{n}$, we get $a = 0$, hence $t = 1$.

\item
If $a + 1 \equiv li\pmod{n}$, then $A = \{ a + 2, t \}$ and $B = \{ 1, a + t + 1 \}$. If $a + 2 \equiv t \pmod{n}$ and $a + 1 + t \equiv 1 \pmod{n}$, we obtain $2 \equiv 0 \pmod{n}$. Otherwise, the elements of $A$ and $B$ are distinct modulo $n$, so we have either $t \equiv 1 \pmod{n}$, or $t \equiv a + 1 + t \pmod{n}$ and $a + 2 \equiv 1 \pmod{n}$, which implies $a \equiv -1 \pmod{n}$. Combined with $2t \equiv a + i \pmod{n}$, and $2t \equiv -1 \pmod{n}$, this implies $i \equiv 0 \pmod{n}$.

\item
If $t \equiv li\pmod{n}$, then $A = \{ t + 1, a + 1 \}$ and $B = \{ 2t, 1 \}$. If the set elements are not distinct, then we have $t \equiv a \pmod{n}$ and $2t \equiv 1 \pmod{n}$; combined with $2t \equiv -1 \pmod{n}$, the latter clearly leads to a trivial case. Otherwise, we have $t + 1 \equiv 1 \pmod{n}$ or $t + 1 \equiv 2t \pmod{n}$, both of which imply $t \equiv 1 \pmod{n}$.
\end{itemize}

\noindent\textbf{Case 4.3.B:} If $li \equiv -t\pmod{n}$, then $A = \{ a + 1, t, -t \}$ and $B = \{ 1, 2t + 1 \}$. Two of the elements of $A$ must coincide modulo $n$, and this clearly can not be $t$ and $-t$. If $a + 1 \equiv t \pmod{n}$, then $A = \{ t+1, -t \}$ while $B = \{ 2t + 1, 1 \}$, leading to trivial cases. If $a + 1 \equiv -t \pmod{n}$, then $A = \{ 1-t, t \}$ and $B = \{ 2t + 1, 1 \}$, which ultimately has the same effect.

\noindent\textbf{Case 4.3.C:} If $2t + 1 \equiv li + t \pmod{n}$, i.e. $li \equiv t + 1 \pmod{n}$, then the sets become $A = \{ a+1, t, t+1 \}$ and $B = \{ 2t + 2, 0 \}$. Two of the elements in $A$ must coincide.

\begin{itemize}
    \item If $a + 1 \equiv t \pmod{n}$, then $A = \{ t + 2 \}$, so we must have $2t + 2 \equiv 0 \pmod{n}$, implying $t \equiv -1 \pmod{n}$.

    \item If $a + 1 \equiv t + 1 \pmod{n}$, i.e. $a \equiv t \pmod{n}$, then $A = \{ t+2, t \}$ and $B = \{ 2t + 2, 0 \}$. Then $t \equiv -2 \pmod{n}$ combined with $2t + 1 \equiv 0 \pmod{n}$ implies $3 \equiv 0 \pmod{n}$, while $t+2 \equiv 2t + 2 \pmod{n}$ and $2t + 2 \equiv 0 \pmod{n}$ imply trivial cases.

    \item If $t \equiv t + 1 \pmod{n}$, we immediately get $1 \equiv 0 \pmod{n}$.
\end{itemize}

\noindent\textbf{Case 4.3.D:} If all elements in the sets are distinct, then comparing the sums of their elements modulo $n$ yields $2t \equiv a \pmod{n}$. Then $A = \{ 2t + 1, t, li \}$ and $B = \{ 2t + 1, li + t, 0 \}$. Unless $t \equiv 0 \pmod{n}$, we must then have $li \equiv 0 \pmod{n}$, from where it is easy to derive $i = 0$.

\noindent\textbf{Case 4.4:} If $t \equiv a$, then
\begin{align*}
    A = \{ 2t + li, t + 1, li \},\
    B = \{ 2t + 1, li + t, 0 \}.
\end{align*}

\noindent\textbf{Case 4.4.A:} If $2t + li \equiv t + 1\pmod{n}$, i.e. $t \equiv 1 - li\pmod{n}$, then $A = \{ t + 2, 1-t \}$ and $B = \{ 2t + 1, 1, 0 \}$. Two of the elements of $B$ must now collide, and the only non-trivial possibility is $2t + 1 \equiv 0 \pmod{n}$. Then $B$ collapses to $\{ 2 \}$, and $A$ must collapse to $\{ t + 3 \}$, hence $t \equiv 1 \pmod{n}$.

\noindent\textbf{Case 4.4.B:} If $2t + li \equiv li\pmod{n}$, we immediately get a contradiction.

\noindent\textbf{Case 4.4.C:} If $t + 1 \equiv li\pmod{n}$, then $A = \{ 3t + 1, t + 2 \}$ and $B = \{ 2t+2, 0 \}$. Now, $2t + 2 \equiv 0\pmod{n}$ implies $t \equiv -1\pmod{n}$, while $t + 2 \equiv 2t + 2\pmod{n}$ implies an immediate contradiction. We must therefore have $t \equiv -2\pmod{n}$, but then $3t + 1 \equiv 2t + 2\pmod{n}$ implies $t \equiv 1\pmod{n}$.

\noindent\textbf{Case 4.4.D:} If all elements are distinct, then the sum of the elements in $A$ and $B$ must be congruent modulo $n$, i.e. $3t + 2li + 1 \equiv 3t + li + 1 \pmod{n}$, implying $li \equiv 0 \pmod{n}$. Then $A = \{ 2t, t + 1, 0 \}$ and $B = \{ 2t + 1, t, 0 \}$, leading to trivial cases.

\noindent\textbf{Case 4.5:} If $t \equiv li\pmod{n}$, then we have
\begin{align*}
    A = \{ 3t, t+1, a \},\
    B = \{ 2t + 1, 2t, 0 \}.
\end{align*}

\noindent\textbf{Case 4.5.A:} If $2t \equiv -1\pmod{n}$, then $A = \{ t-1, t+1, a \}$ and $B =\{ 2t,1\}$.

\begin{itemize}
    \item
If $t-1 \equiv t+1\pmod{n}$, then we get $2 \equiv 0\pmod{n}$.

\item
If $t-1 \equiv a\pmod{n}$, then we have $A = \{ a+1, t+1 \}$. If $a \equiv t\pmod{n}$, then $1 \equiv 0\pmod{n}$. If $t+1 \equiv 1\pmod{n}$, we get a contradiction, and if $t+1 \equiv 2t\pmod{n}$, then $t \equiv 1\pmod{n}$.

\item
If $t + 1 \equiv a\pmod{n}$, then $A = \{ t-1, t+2 \}$ and $B = \{ 2t,1 \}$, so we get either $3 \equiv 0\pmod{n}$, or $t \equiv -1$, or $t \equiv 2\pmod{n}$ which, combined with $2t \equiv -1\pmod{n}$ yields $5 \equiv 0\pmod{n}$.

\end{itemize}

\noindent\textbf{Case 4.5.B:} If $2t \equiv 0\pmod{n}$, we get a contradiction immediately.

\noindent\textbf{Case 4.5.C:} If $2t + 1 \equiv 2t\pmod{n}$, we get $1 \equiv 0\pmod{n}$.

\noindent\textbf{Case 4.5.D:} If all elements are distinct, then we deduce $a = 0$ and then $2t \equiv t+1\pmod{n}$, so that $t \equiv 1\pmod{n}$.

\noindent\textbf{Case 4.6:} If $li \equiv a\pmod{n}$, then
\begin{align*}
    A = \{ 2t + li, t, li + 1 \},\
    B = \{ 2t+1, li+t, 0 \}.
\end{align*}

\noindent\textbf{Case 4.6.A:} If $li + t \equiv 0\pmod{n}$, i.e $li \equiv -t \equiv a\pmod{n}$, then $A = \{ t + 1, 1-t \}$ and $B = \{ 2t + 1, 1 \}$, which reduces to trivial cases.

\noindent\textbf{Case 4.6.B:} If $2t + 1 \equiv 0\pmod{n}$, i.e. $2t \equiv -1\pmod{n}$, then $A = \{ li - 1, t, li + 1 \}$ and $B = \{ 1, li + t \}$.

\begin{itemize}
\item 
If $li - 1 \equiv t$, then $A = \{ t + 1, li + 1 \}$ and $B = \{ 1, li + t \} = \{ li - t, li + t \}$. The elements must clearly be distinct modulo $n$, and then we either have $t + 1 \equiv 1 \pmod{n}$, or $li \equiv 1 \pmod{n}$ and $li \equiv 0 \pmod{n}$.

\item
If $li - 1 \equiv li + 1\pmod{n}$, we get $2 \equiv 0\pmod{n}$ immediately.

\item
If $t \equiv li + 1\pmod{n}$, then $A = \{ li-1, t+1 \}$ and $B = \{ 1,li+t \} = \{ t-li, li+t \}$. If the elements in $A$, respectively $B$, coincide modulo $n$, then we get $li \equiv 0 \pmod{n}$ and hence $t \equiv 1 \pmod{n}$. Otherwise, we get either $t + 1 \equiv 1 \pmod{n}$, or $t + 1 \equiv li + t \pmod{n}$, i.e. $li \equiv 1 \pmod{n}$ so that $t \equiv 2 \pmod{n}$ and hence $3 \equiv 0 \pmod{n}$.
\end{itemize}

\noindent\textbf{Case 4.6.C:} If $2t + 1 \equiv li + t\pmod{n}$, so that $li \equiv t + 1\pmod{n}$, then we have $A = \{ 3t + 1, t, t + 2 \}$ and $B = \{ 2t + 2, 0 \}$. Now, if $3t + 1 \equiv t\pmod{n}$, we get $A = \{ t + 1, t + 2 \}$ and $B = \{ 1, 0 \}$, which clearly leads to trivial cases. If $3t + 1 \equiv t + 2\pmod{n}$, then we get $2t \equiv 1$, and so $A = \{ t + 3, t \}$ and $B = \{ 3, 0 \}$ which is again easy to handle. Finally, if $t \equiv t + 2\pmod{n}$, we get $2 \equiv 0\pmod{n}$.

\noindent\textbf{Case 4.6.D:} If all elements in $A$, respectively,  $B$ are distinct, then their sums must be congruent modulo $n$, whence $3t + 2li + 1 \equiv 3t + li + 1 \pmod{n}$, so that $li \equiv 0 \pmod{n}$. Then $A = \{ 2t, t, 1 \}$ and $B = \{ 2t + 1, t, 0 \}$, yielding only trivial cases.

\noindent\textbf{Case 5:} If $2t \equiv li + t\pmod{n}$, i.e. $t \equiv li\pmod{n}$, then
\begin{align*}
    A = \{ 3t, t+1, a \},\
    B = \{ 2t+1, a +i, 0 \}.
\end{align*}

\noindent\textbf{Case 5.1:} If $2t + 1 \equiv 0\pmod{n}$, then
\begin{align*}
    A = \{ t-1, t+1, a \},\
    B = \{ 1, a+i \}.
\end{align*}

Two of the elements in $A$ must now collide.

\noindent\textbf{Case 5.1.A:} If $t - 1 \equiv t + 1 \pmod{n}$, then $2 \equiv 0\pmod{n}$.

\noindent\textbf{Case 5.1.B:} If $t - 1 \equiv a\pmod{n}$, then $A = \{ t, t+1 \}$ and $B = \{ 1, a+i \}$. If $t \equiv t+1$, then $1 \equiv 0\pmod{n}$; if $t \equiv 1\pmod{n}$, we get the Gold case; if $t \equiv a+i\pmod{n}$, then from $t + 1 \equiv 1 \pmod{n}$ we get $t \equiv 0 \pmod{n}$.

\noindent\textbf{Case 5.1.C:} If $t+1 \equiv a\pmod{n}$, then $A = \{ t-1, a+1 \}$ and $B = \{ 1, a+i \}$. If $t - 1 \equiv a + 1\pmod{n}$, then $a - 2 \equiv a + 1\pmod{n}$ and so $3 \equiv 0\pmod{n}$. If $t-1 \equiv 1\pmod{n}$, then $t \equiv 2\pmod{n}$ and we get $4 \equiv -1\pmod{n}$. If $t-1 \equiv a+i\pmod{n}$ and $a + 1 \equiv 1\pmod{n}$, then $a = 0$, and so $t \equiv -1\pmod{n}$.

\noindent\textbf{Case 5.2:} If $a + i \equiv 0\pmod{n}$, then
\begin{align*}
    A = \{ 3t, t+1, a \},\
    B = \{ 2t + 1, 1 \}.
\end{align*}

\noindent\textbf{Case 5.2.A:} If $3t \equiv t + 1\pmod{n}$, i.e. $2t \equiv 1\pmod{n}$, then $A = \{ t + 2, a \}$ and $B = \{ 1, 2 \}$, so either $1 \equiv 2\pmod{n}$, $t \equiv 0\pmod{n}$ or $t \equiv -1\pmod{n}$.

\noindent\textbf{Case 5.2.B:} If $3t \equiv a\pmod{n}$, then $A = \{ a + 1, t + 1 \}$ and $B = \{ 2t + 1, 1 \}$; once again, we get $t \equiv 0\pmod{n}$.

\noindent\textbf{Case 5.2.C:} If $t+1 \equiv a\pmod{n}$, then $A = \{ 3t, a+1 \}$ and $B = \{ 2t + 1, 1 \}$. If $2t + 1 \equiv 1\pmod{n}$, we get a contradiction. If $3t \equiv 2t + 1\pmod{n}$, we get $t \equiv 1\pmod{n}$. If $3t \equiv 1\pmod{n}$ and $a + 1 \equiv 2t + 1\pmod{n}$, then $a \equiv 2t \equiv t + 1\pmod{n}$ and so $t \equiv 1\pmod{n}$.

\noindent\textbf{Case 5.3:} If $2t + 1 \equiv a+i\pmod{n}$, then
\begin{align*}
    A = \{ 3t, t+1, a \},\
    B = \{ 2t + 2, 0 \}.
\end{align*}

\noindent\textbf{Case 5.3.A:} If $3t \equiv t + 1\pmod{n}$, i.e. $2t \equiv 1\pmod{n}$, then $A = \{ t + 2, a \}$ and $B = \{ 3, 0\}$. Unless $3 \equiv 0\pmod{n}$ or $t \equiv 1\pmod{n}$, then $t \equiv -2\pmod{n}$ and $a \equiv 3\pmod{n}$. From $2t \equiv 1\pmod{n}$ we get $-4 \equiv 1\pmod{n}$.

\noindent\textbf{Case 5.3.B:} If $3t \equiv a\pmod{n}$, then $A = \{ 3t + 1, t + 1 \}$ and $B = \{ 2t + 2, 0 \}$. If $3t + 1 \equiv t+1\pmod{n}$, we get $t \equiv 0\pmod{n}$. If $t + 1 \equiv 2t + 2\pmod{n}$, then $t \equiv -1\pmod{n}$. If $t + 1 \equiv 0\pmod{n}$, then $t \equiv -1\pmod{n}$.

\noindent\textbf{Case 5.3.C:} If $t + 1 \equiv a\pmod{n}$, then $A = \{ 3t, t+2 \}$ and $B = \{ 2t + 2, 0 \}$. If $3t \equiv t + 2\pmod{n}$, then $t \equiv 1\pmod{n}$. If $3t \equiv 2t + 2\pmod{n}$, we get $t = 2$ and from $t + 2 \equiv 0\pmod{n}$ we get $4 \equiv 0\pmod{n}$. If $3t \equiv 0\pmod{n}$, we get a contradiction immediately.

\noindent\textbf{Case 6:} If $li + t \equiv a + i\pmod{n}$, then
\begin{align*}
    A = \{ 2t + li, t, li, a \},\
    B = \{ li + t + 1, 2t, 0 \}.
\end{align*}
This is treated similarly to all the previous cases.

\noindent\textbf{Case 7:} If all elements of $A$, respectively,  $B$ are distinct modulo, then by Lemma~\ref{lemUniqueExpansion} we have
\begin{equation*}
    \{ 2t + li, t, li, a \} \equiv \{ 2t, li+t, a+i, 0 \} \pmod{n}.
\end{equation*}

\begin{itemize}
\item Clearly, $0 \equiv t\pmod{n}$ and $t \equiv 2t\pmod{n}$ is impossible.

\item
If $t \equiv li + t\pmod{n}$, then $li \equiv 0\pmod{n}$, and so $A = \{ 2t, t, 0, a \}$ and $B = \{ 2t, t, a+i, 0 \}$. Then $a \equiv a + i\pmod{n}$ so $i \equiv 0\pmod{n}$ which can not happen.

\item
If $t \equiv a + i\pmod{n}$, then $A = \{ 2t + li, t, li, a \}$ and $B = \{ 2t, li+t, t, 0 \}$. 
\begin{itemize}
    \item  If $li \equiv 0\pmod{n}$, then $2t + li \equiv 2t\pmod{n}$, so $a \equiv li + t\pmod{n}$ and $a \equiv li + a + i\pmod{n}$ implies $i \equiv 0\pmod{n}$.

    \item
    If $li \equiv li + t\pmod{n}$, we get $t \equiv 0\pmod{n}$.

    \item
    If $li \equiv 2t\pmod{n}$, then $A = \{ 4t, t, 2t, a \}$ and $B = \{ 2t, 3t, t, 0 \}$ which necessarily implies $t \equiv 0\pmod{n}$.
\end{itemize}
\end{itemize}

This completes the proof.
\end{proof}

We can verify that for $n = 3$ (where the Kasami exponents coincide with the Gold exponents) and for $n = 5$, we can express the Kasami using $e(l,k)$. In the case of $n = 5$, the relevant exponents are $5 = e(2,2)$, $7 = (3,1)$, $9 = e(2,3)$, and $25 \equiv e(3,4) \pmod{2^5-1}$. By the above characterizations, the Kasami family never coincides with $e(l,k)$ in any other cases.

\subsection{Niho even case}

The Niho exponent for $n = 2t+1$ with $t$ even is $2^t + 2^{t/2} - 1$. The algebraic degree is $(t+2)/2$. 

\begin{theorem}
\label{propNiho_even}
Let $t > 2$ be an even natural number, $n = 2t + 1$ and $i$ be such that $\gcd(i,n) = 1$. Then the even Niho exponent $2^{t}+2^{t/2}-1$ is never cyclotomic equivalent to $e((t+2)/2,i)$ over $\F_{2^n}$.
\end{theorem}
\begin{proof}
Once again, suppose that
\begin{equation*}
	2^a(2^t + 2^{t/2} - 1) \equiv e( (t+2)/2, i)\pmod{2^n-1}
\end{equation*}
for some natural number $i$. We split the proof into several cases, applying Lemma \ref{lemUniqueExpansion}.

\noindent\textbf{Case 1:} If $a + t < n$,  the (shifted by $a$) Niho exponent is then
\begin{equation*}
	2^{a + t} + 2^{a + t/2} - 2^a = 2^{a+t} + \sum_{j = a}^{a+t/2-1} 2^j,
\end{equation*}
and so the set of its exponents is precisely
\begin{equation*}
	A = \{ a, a+1, a+2, \ldots, a + t/2 - 1 \} \cup \{ a + t \}.
\end{equation*}
The set of exponents of $e( (t+2)/2, i)$ is
\begin{equation*}
	B =  \{ 0, i, 2i, 3i, \ldots, (t/2) i \} \Mod{n}.
\end{equation*}
If $0 = a + j$ for $j > 0$, then we immediately get a contradiction. Thus, we must have $a = 0$ and the set $A$ becomes
\begin{equation*}
A = \{ 0, 1, 2, \ldots, t/2 - 1 \} \cup \{ t \}.
\end{equation*}
Let $q = t/2$. The sum of all elements in $A$ is
\begin{equation*}
	t + \sum_{j = 1}^{t/2-1} j = 2q + \frac{(q-1)q}{2} = \frac{q^2 + 3q}{2},
\end{equation*}
while the sum of all elements in $B$ is
\begin{equation*}
	i \frac{q(q+1)}{2} = \frac{(q^2 + q)i}{2}
\end{equation*}
modulo $n$. Since $2,q$ are invertible modulo $n$, we must have
\begin{equation*}
	q + 3 \equiv (q+1)i.
\end{equation*}
Multiplying both sides by $4$, we have
\begin{equation*}
	4q + 12 \equiv (4q + 4)i \pmod{n}.
\end{equation*}
Since $n = 2t+1 = 4q + 1$, the above becomes
\begin{equation*}
	11 \equiv 3i \pmod{n}.
\end{equation*}

If $q \equiv 2 \pmod{3}$, then $3 \mid n$, and so $3i$ does not have an inverse whereas $11$ does, which is a contradiction.

If $q \equiv 1 \pmod{3}$, then the inverse of $3$ modulo $n = 4q + 1$ is $(4q+2)/3$. From the identity $11 \equiv 3i \pmod{n}$ above, we thus get
\begin{equation*}
  i \equiv \frac{11(4q+2)}{3} \pmod{n},
\end{equation*}
which becomes
\begin{equation*}
  i \equiv 3(4q+2) + \frac{2(4q+2)}{3} \equiv 3 + \frac{8q + 4}{3} \equiv \frac{8q + 13}{3} \equiv \frac{2n + 11}{3} \pmod{n}.
\end{equation*}
If $(2n + 11)/3 < n$, then no modulation is necessary, and this number belongs to the set $A$; it must thus either be no greater than $t/2 - 1$, or it must equal $t$. In the former case, we have
\begin{equation*}
  \frac{2n + 11}{3} \le \frac{t-2}{2} \iff 4n + 22 \le 3t - 6 \iff 3t \ge 4n + 28 \ge 4n,
\end{equation*}
so that $t \ge 4/3 n$, which contradicts the choice of $t$. In the latter case, we have
\begin{equation*}
  \frac{2n + 11}{3} = t \iff 2n + 11 = 3t,
\end{equation*}
but since $n = 2t + 1$, we have $4t + 13 = 3t$, i.e. $t = -13$, which cannot be.

The only remaining case is when $(2n + 11)/3 \ge n$, i.e. $n \le 11$. Since we consider dimensions of the form $n = 2t + 1$ for even $t$ with $t/2 \equiv 1 \pmod{3}$, this leaves only $n = 5$. In this case, we can see that $e(2,2)$ is exactly the Niho exponent $2^2 + 2^1 - 1 = 5$.

Finally, suppose that $q \equiv 0 \pmod{3}$, i.e. $3 \mid q$. The inverse of $3$ modulo $n = 4q + 1$ is
\begin{equation*}
  3^{-1} \equiv \frac{8q + 3}{3} \pmod{n},
\end{equation*}
and so we get
\begin{equation*}
  \begin{split}
    i & \equiv 11 \cdot 3^{-1} \equiv \frac{11(8q+3)}{3} \equiv 3(8q+3) + \frac{16q + 6}{3} \\
    & \equiv 3 + \frac{16q + 6}{3} \equiv \frac{16q + 15}{3} \equiv \frac{n + 11}{3} \pmod{n}.
  \end{split}
\end{equation*}

If $(n+11)/3 < n$, then it must be contained in $A$. If
\begin{equation*}
  \frac{n+11}{3} \le \frac{t}{2}-1,
\end{equation*}
then we get
\begin{equation*}
  3t \ge 2n + 28,
\end{equation*}
which is clearly impossible.

If $(n+11)/3 = t$, then we get $n + 11 = 3t$, i.e. $t = 12$, so that $n = 25$. We can verify by exhaustive search that the Niho exponent for $n = 25$ cannot be expressed using $e(t,i)$ for any choice of the parameters $t$ and $i$.

Finally, if $(n+11)/3 \ge n$, then we have $2n \le 11$, and so $n \le 6$; the only possibility is $n = 1$, i.e. $t = 0$, which is not a valid choice for the Niho family.

\noindent\textbf{Case 2:} If $a + t/2 < n \le a+t$, then let $k = a + t - n$. We have $k < a + t/2$ since if $k = a + t -n \ge a + t/2$, then $t/2 \ge n$. Similarly, we must have $k < a$ since $a + t - n = k \ge a$ means $t \ge n$. The Niho exponent is thus
\begin{equation*}
  2^{a + t/2} - 2^a + 2^k = \left( \sum_{j = a}^{a + t/2 - 1} 2^j \right) + 2^k
\end{equation*}
in this case. The set of exponents is
\begin{equation*}
  A = \{ k \} \cup \{ a, a+1, a+2, \dots, a + t/2 - 1 \},
\end{equation*}
while that of $e(t/2, i)$ is once again
\begin{equation*}
  B = \{ 0, i, 2i, \dots, (t/2)i \} \Mod{n}.
\end{equation*}
As before, we can only have $0 = a$ or $0 = k$. If $a = 0$, then $k = t - n$, i.e. $n = t - k$ which cannot happen since $t < n$ by the hypothesis. Thus, we must have $k = 0$, i.e. $a + t = n$. From here, we can express $a = t + 1$ due to $n = 2t + 1$. We now have
\begin{equation*}
  A = \{ 0, t+1, t+2, \dots, t + t/2 \}.
\end{equation*}
There must exist $0 \le \alpha, \beta \le t/2$ such that $\alpha i \equiv t+1 \pmod{n}$ and $\beta i \equiv t+2 \pmod{n}$. Then either $\alpha - \beta$ or $\beta - \alpha$ is in the range $\{ 0, 1, \dots, t/2 \}$, and so either $(\alpha - \beta)i \equiv -1 \pmod{n}$ or $(\beta - \alpha)i \equiv 1 \pmod{n}$ must be in $A$. In other words, either $2t$ or $1$ must belong to $A$, which is clearly impossible for $t > 1$. We have thus reached a contradiction.

\noindent\textbf{Case 3:} If $a + t/2 \ge n$, then let $k = a + t/2 - n$. We must have $k + t/2 < a$ since $k + t/2 \ge a$ implies $a + t/2 - n + t/2 \ge a$, i.e. $t \ge n$; and so we have $a > k + t/2 > k$, and the Niho exponent becomes
\begin{equation*}
  -2^a + 2^{k + t/2} + 2^k = - (2^a - 2^{k + t/2}) + 2^k.
\end{equation*}
Using Observation~\ref{obsBinaryComplement}, we can see that
\begin{equation*}
  -(2^a - 2^{k+t/2}) \equiv \sum_{j = 0}^{k + t/2 - 1} 2^j + \sum_{j = a}^{n-1} 2^j \pmod{n},
\end{equation*}
and then
\begin{equation*}
  -(2^a - 2^{k + t/2}) + 2^k \equiv \left( \sum_{j = 0}^{k-1} 2^j \right) + 2^{k + t/2} + \left( \sum_{j = a}^{n-1} 2^j \right) \pmod {n};
\end{equation*}
consequently, we obtain the set of exponents
\begin{equation*}
  A = \{ 0,1,2,\dots,k-1 \} \cup \{ k + t/2 \} \cup \{ a, a + 1, \dots, n-1 \},
\end{equation*}
while, as before, the set of exponents corresponding to $e(t/2,i)$ is
\begin{equation*}
  B = \{ 0, i, 2i, \dots, t/2 i \} \pmod{n}.
\end{equation*}
We must have $\alpha i \equiv 1 \pmod{n}$ and $\beta i \equiv n-1 \pmod{n}$ for some $1 \le \alpha, \beta \le t/2$. Then $(\alpha + \beta)i \equiv 0 \pmod{n}$, and since $\gcd(i,n) = 1$ by the hypothesis, we get $\alpha + \beta \equiv 0 \pmod{n}$, i.e. $n \mid \alpha + \beta$. But since $1 \le \alpha, \beta \le t/2 = (n-1)/4$, this is impossible.

We have thus shown our claim (for $n=5$, we confirmed it computationally). 
\end{proof}
We now concentrate on the inverse even Niho exponent.
\begin{theorem}
    Let $t > 2$ be an even natural number, $n = 2t + 1$ and $i$ be such that $\gcd(i,n) = 1$. Then the inverse of the even Niho exponent $(2^{t}+2^{t/2}-1)^{-1} \pmod{2^n-1}$ is never cyclotomic equivalent to $e((t+2)/2,i)$ over $\F_{2^n}$.
  \label{propNiho_even_inv}
\end{theorem}
\begin{proof}
 We use~\cite[Lemma 3]{BCCDK22}, representing the even Niho exponent power function as the composition of $x^3$ and the inverse of a cubic power function. Precisely, 
\[
2^t+2^{t/2}-1\equiv 2^{3t/2} \frac{3}{2^{t+1}+2^{t/2}+1} \pmod {2^{2t+1}-1}.
\]
To investigate the cyclotomic equivalence of $e(l,k)$, $l<n$ with the even Niho inverse, it is sufficient to consider the congruence ($n=2t+1$, $t$ even)
\[
3 e(l,k)\equiv 2^a(2^{t+1}+2^{t/2}+1) \pmod {2^n-1}
\]
for some positive integers $a<n,l<n$. This is equivalent to 
\begin{equation}
    \label{eq:Niho_even}
\sum_{i=0}^{l-1} 2^{ki}+\sum_{i=0}^{l-1} 2^{ki+1}\equiv 2^{a+t+1}+2^{t/2+a}+2^a\pmod{2^n-1}.
\end{equation}

We start with $k=1$. Equation~\eqref{eq:Niho_even} becomes
\[
1+\sum_{j=2}^{l-1} 2^j+2^{l+1}\equiv 2^{a+t+1}+2^{t/2+a}+2^a\pmod{2^n-1}.
\]
If $l+1=n$, the sets of exponents above are
\begin{align*}
A&=\{1,2,\ldots,l-1\},\text{ all smaller than } n\\
B&=\{t+a+1,t/2+a,a\}\pmod n.
\end{align*}
Thus, $l\leq 4$. First, we take $l=4$.
If $a=1$, then $B=\{1,t/2+1,t+2\}=A=\{1,2,3\}$, which cannot happen (recall that $t=(n-1)/2\geq 2$).
If $t/2+a\equiv 1\pmod n$ (recall that $a<n$), then $a=n+1-t/2=n+1-(n-1)/4=(3n+5)/4$.
The sets are now
$B=\{1,t+a+1=\frac{5n+7}{4}\equiv \frac{n+7}{4}\pmod n,a=\frac{3n+5}{4} \}=\{1,2,3\}$, and that is impossible.
The case of $t+a+1\equiv 1\pmod n$ implies $t+a\equiv 0\pmod n$, and so, $a=n-t=t+1=\frac{n+1}{2}$. Thus, $B=\{1,\frac{n+1}{2},\frac{3n+1}{4}\}=\{1,2,3\}$, which cannot happen. We can theoretically argue it, but to simplify the argument, if $l=2,3$ (hence $n=3,4$), we checked computationally that our congruence~\eqref{eq:Niho_even} is impossible.

If $l+1<n$, then the set $A=\{0,2,\ldots,l-1,l+1\}$ contains only distinct exponents and so, $l=3$ and $A=\{0,2,4\}$. First, $a=0$ is impossible. If $t+a+1\equiv 0\pmod n$, then $a=n-t-1=\frac{n-1}{2}=t$, and so, $B=\{0,a=\frac{n-1}{2},t/2+a=\frac{3(n-1)}{4}\}$, which cannot be equal to~$A$. If $t/2+a\equiv 0\pmod n$, then $a=n-t/2=\frac{3n+1}{4}$, and so,
$B=\{0,a=\frac{3n+1}{4},t+a+1\equiv\frac{n+3}{4}\pmod n \}$, which can equal $A$ if and only if $n=5,a=4,l=3$, which is possible and the congruence becomes $3\cdot (2^2+2+1)\equiv 2^4 \cdot (2^3+2+1)\pmod {2^5-1}$.

We now let $k>1$. The set of exponents in the congruence~\eqref{eq:Niho_even}  are
\begin{align*}
A&=\{0,1,k,k+1,2k,2k+1,\ldots,(l-1)l,(l-1)k+1\}\Mod n,\\
B&=\{a,t+a+1,t/2+a\}\Mod n.
\end{align*}
For the set $A$ to compress we need to have $s_1<s_2<n$ such that either $s_2k\equiv s_1k\pmod n$, which is impossible since $\gcd(k,n)=1$, $s_2k\equiv s_1k+1\pmod n$, or $s_2k+1\equiv s_1k\pmod n$. These last two cases are treated similarly, so we only deal with the first one. We take $s_2k\equiv s_1k+1\pmod n$ and $s_2$ smallest with this property. Thus, $(s_2-s_1)k\equiv 1\pmod n$, $s_2-s_1<n$. Surely, the two elements $s_1k+1,s_2k$ compress into $s_1k+2$, which occurs by itself in $A$, since $k>1$, so no compression occurs. Now, the same will happen for all the remaining exponents above $s_2k+1$, since $(s_2+j)k\equiv (s_1+j)k+i\pmod n$. If $s_1>0$, the set $A$ cannot compress to only three exponents as in the set $B$, so we must have $s_1=0,k=n-1$. It follows that $A=\{0,1, n-1,n,2(n-1),2(n-1)+1\}\pmod n=\{0,2,3\}$, which renders the case $n=5,k=4,l=3,a=2$, that is, the congruence $3\cdot(2^8+2^4+1)\equiv 2^2 (2^3+2+1)\pmod {2^5-1}$.

We thus have our claim (we computationally checked that the two cases do happen for the inverse of the Niho exponent).
\end{proof}

To conclude, we observe that for $t = 2$, i.e. $n = 5$, the Niho exponent or its inverse can be equivalent to $e(l,k)$. These are, in fact, precisely the exponents for $n = 5$ that coincide with the Kasami family. 

\subsection{Niho odd case}

In this case, the exponent is of the form $2^t + 2^{(3t+1)/2} - 1$ for $t$ odd, with $n = 2t + 1$. The algebraic degree is $t + 1$. 

\begin{theorem}
Let $t > 1$ be an odd natural number, $n = 2t + 1$ and $k$ be such that $\gcd(k,n) = 1$. Then the odd Niho exponent $2^{t}+2^{\frac{3t+1}{2}}-1$ can never be in the cyclotomic coset of $e(l,k)$ over $\F_{2^n}$.
  \label{propNiho_odd}
\end{theorem}
\begin{proof}
Suppose that there exists some $a \le n-1$ such that $2^t + 2^{\frac{3t+1}{2}} - 1$ is congruent with $2^ae(l,k)$ for some $l<n,1\leq k$. From Lemma~\ref{lem2}, we can assume that $1\leq k\leq \frac{n+1}{2}$.  Writing $e(l,k)=\frac{2^{lk}-1}{2^k-1}$ and multiplying throughout by $2^k-1$, we get
\[
2^{k+t}+2^{\frac{3t+1}{2}+k}-2^k-2^t-2^{\frac{3t+1}{2}}+1\equiv 2^{lk+a}-2^a \pmod {2^n-1},
\]
that is
\begin{equation}
  2^{\frac{3t+1}{2}+k}+2^{k+t} +2^a+1\equiv
  2^{lk+a}+2^k+2^t+2^{\frac{3t+1}{2}}\pmod{2^n-1}.
\end{equation}

If $a=0$, the sets of exponents are
\[
A=\left\{1,k+t,\frac{3t+1}{2}+k\right\}\Mod n, B=\left\{k,t,\frac{3t+1}{2},lk\right\}\Mod n.
\]
We know that $k+t\leq n$ (since $k\leq \frac{n+1}{2}$ and $t=\frac{n-1}{2}$). If $k+t=n$, that is, $k=\frac{n+1}{2}$, then 
$A=\{0,1, \frac{n+1}{4}\}$ and $B=\{\frac{n-1}{2}, \frac{n+1}{2},\frac{3n-1}{4},\frac{l(n+1)}{2}\}$. It follows that $l(n+1)\equiv l\equiv 0\pmod n$, but that is impossible, since $1<l<n$.

Next, let $k+t<n$ and assume that $A$ does not compress modulo $n$ (that ultimately means that $\frac{3t+1}{2}+k\not\equiv 1\pmod n$, since $n>k+t>1$). If $k=1$, the two sets of exponents are $\{1,t+1,\frac{3t+3}{2}\}$ and $\{1,t,\frac{3t+1}{2},l\}$, which cannot possibly be equal. If $t=1,k>1$, then the two sets become $A=\{1,k+1,k+2\}$, $B=\{1,k,2,l\}$, which is not possible.
If $lk\equiv 1\pmod n$, the two sets become $\{1,\frac{3n-1}{4},0\}$ and $\{\frac{n+1}{4},\frac{n-1}{2},1\}$, which yet again is not possible (since $0$ cannot be equated to anything in~$B$).

Next, we assume that $A$ compresses, that is, $\frac{3t+1}{2}+k\equiv 1\pmod n$. Thus, $k=\frac{n+5}{4}$, $k+t\geq 3$ and so,
$A=\{2,\frac{3(n+1)}{4}\}$, $B=\{\frac{n+5}{4},\frac{n-1}{2},\frac{3n-1}{4},l\frac{n+5}{4}\Mod n \}$. Going through the possibilities (for $n>3$), we see that there are no values of $n$ for which the two sets match. 

We next assume that $a>0$.
If $lk+a<n$ and $\frac{3t+1}{2}+k<n$, then the sets of exponents must be the same (without modulation), but that is impossible since $0$ cannot be any of the exponents $k,t,\frac{3t+1}{2},lk+a$. Therefore, either  $lk+a\geq n$ or $\frac{3t+1}{2}+k\geq n.$ If $\frac{3t+1}{2}+k<n$, but $lk+a\geq n$, then $lk+a\equiv 0\pmod n$ (observe that $A$ cannot remove 0 by possible compression, since $k+t<\frac{3t+1}{2}+k<n$ and either of the cases $a=k+t=n-2, \frac{3t+1}{2}+k=n-1$, when $2^a+2^{k+t}+2^{\frac{3t+1}{2}+k}=2^n\equiv 2^0\pmod {2^n-1}$, or $a=\frac{3t+1}{2}+k=n-2, k+t=n-1$, or $a=\frac{3t+1}{2}+k=n-1$, or $a=k+t=n-1$,  will all render contradictions) and the sets of exponents become
\begin{align*}
A=\left\{0,a,k+t,\frac{3t+1}{2}+k\right\}, B=\left\{0,k,t,\frac{3t+1}{2}\right\}.
\end{align*}
Surely, the only possibility is for $k$ to be equal to $a$, and the same is true for $t$, so that $k=a=t$, but then $\frac{3t+1}{2}+k\geq 2t+1$, a contradiction.
 
If $\frac{3t+1}{2}+k\geq n$ and $lk+a< n$, since $k\leq \frac{n+1}{2}$, then $\frac{3t+1}{2}+k\leq \frac{5n+1}{4}=n+\frac{n+1}{4}$, and so, we must have $\frac{3t+1}{2}+k=n$ (otherwise, $0$ remains in $A$, and that should not be the case as $B$ cannot contain~$0$), that is, $k=\frac{n+1}{4}$ (so, $k+t<n$). The sets of exponents become now (the two copies of $0$ in $A$ compress to a $1$)
\[
A=\left\{1,a,k+t=\frac{3n-1}{4} \right\}, B=\left\{k,t,\frac{3t+1}{2}=\frac{3n-1}{4},lk+a\right\}.
\]
If $a=1$ (recall that $n>3$ and  $3\leq k+t<n$), then $A=\{1,1,\frac{3n-1}{4}\}=\{2,\frac{3n-1}{4}\}$ and (since all of its elements are smaller than~$n$, then $k=t=1$) $B=\{1,1,\frac{3n-1}{4},l+1\}=\{2,\frac{3n-1}{4},l+1\}$ (by compression and taking modulo $2^n-1$). Surely, that is not possible. If $1<a<n$, then only $k,t, lk\pmod n$ can be $1$, but they all lead to contradiction.

 It remains to look at the case $\frac{3t+1}{2}+k\geq n$ and $lk+a\geq n$. As remarked before, we must have $lk+a\equiv 0\pmod n$, so the two sets of exponents become
\[
A=\left\{0,a,k+t,\left(\frac{3t+1}{2}+k\right)\Mod n\right\}, B=\left\{0,k,t,\frac{3t+1}{2}\right\}.
\]
We observe that $k,t$ are both smaller than  $k+t$ and $\frac{3t+1}{2}+k\leq \frac{3n-1}{4}+\frac{n+1}{2}=n+\frac{n-1}{4}$. Thus, $\left(\frac{3t+1}{2}+k\right)\pmod n=\frac{3t+1}{2}+k -n=k-\frac{n+1}{4}$. The two sets become 
\[
A=\left\{0,a,k-\frac{n+1}{4},k+\frac{n-1}{2}\right\}, 
B=\left\{0,k,\frac{n-1}{2},\frac{3n-1}{4}\right\}.
\]
and so, $k=a$ is the only possibility, as well as, $k+\frac{n-1}{2}=\frac{3n-1}{4}$, and so, $k=\frac{n+1}{4}$. Thus,
\[
A=\left\{ 0,\frac{n+1}{4},0,\frac{3n-1}{4}\right\}, 
B=\left\{0, \frac{n+1}{4},\frac{n-1}{2},\frac{3n-1}{4}\right\}
\]
This is only possible for $n=3$, when the sets compress to $A=B=\{0\}$.

The proof is done.
\end{proof}

\begin{theorem}
    Let $t > 1$ be an odd natural number, $n = 2t + 1$ and $k$ be such that $\gcd(k,n) = 1$. Then the inverse of the odd Niho exponent $(2^{t}+2^{\frac{3t+1}{2}}-1)^{2^n-1}$ can never be cyclotomic equivalent to $e(l,k)$ over $\F_{2^n}$.
  \label{propNiho_odd_inv}
\end{theorem}
\begin{proof}
We use~\cite[Lemma 6]{BCCDK22}, namely, for $t$ odd, we have
\[
2^{\frac{3t+1}{2}}+2^t-1\equiv 2^{\frac{3t-1}{2}}\frac{3}{2^t+2^{\frac{t-1}{2}}+1}\pmod {2^{2t+1}-1}.
\]
As such, it will be sufficient to investigate the congruence (for positive integers $a,k,l<n$),
\begin{equation}
\label{eq:Niho_odd_inv}
3e(l,k)\equiv 2^a(2^t+2^{\frac{t-1}{2}}+1)\pmod {2^n-1}.
\end{equation}
We are going to use some computations done for the inverse even Niho case. If $k=1$, we need to check 
\[
\sum_{i=0}^{l-1} 2^{ki}+\sum_{i=0}^{l-1} 2^{ki+1}\equiv 2^{a+t}+2^{\frac{t-1}{2}+a}+2^a\pmod{2^n-1}.
\]
If $k=1$, the equation becomes
\[
1+\sum_{j=2}^{l-1} 2^j+2^{l+1}\equiv 
2^{a+t}+2^{\frac{t-1}{2}+a}+2^a\pmod{2^n-1}.
\]
If $l+1=n$ (recall that $l<n$), then the corresponding sets of exponents in the congruence are
\begin{align*}
A&=\{1,2,\ldots,l-1\},\text{ all smaller than } n,\\
B&=\{a,(t-1)/2+a,t+a\}\Mod n,
\end{align*}
which implies that $l\leq 4$, and $A=\{1,2,3\}$ (if $A=\{1,2\}$, or $A=\{1\}$, we quickly see that it is not possible). If $a=1$, then $(t-1)/2=1$ and $t=2$, or, $(t-1)/2=2$ and $t=1$, which are both impossible. The other cases can not happen either, since we are dealing with positive integers.

If $l+1<n$, then, as argued before, the set $A$ contains only distinct exponents and so, $l=3$ and $A=\{0,2,4\}$. Thus, $a=0$, $B=\{ 0,(t-1)/2,t\}$, which cannot match $A$.

We next assume that $k>1$. As in the even Niho case,
\begin{align*}
A&=\{0,1,k,k+1,2k,2k+1,\ldots,(l-1)l,(l-1)k+1\}\Mod n,\\
B&=\{a,(t-1)/2+a,t+a\}\Mod n.
\end{align*}
This congruence can be handled using a similar method based on compressing the exponent sets. In all cases, we obtain contradictions or trivial results only.

We thus have the proof of our theorem.
\end{proof}

Once again, we can see that for $t = 1$, i.e. $n = 3$, the odd Niho exponents coincide with the Gold exponents, and this is the only case in which they can be expressed as $e(l,k)$.

\subsection{Dobbertin case}

The Dobbertin exponent is $D_t = 2^{4t} + 2^{3t} + 2^{2t} + 2^t - 1$ for $n = 5t$. Note that in this case we have to consider $k$ with $\gcd(k,n) = 2$ in addition to $\gcd(k,n) = 1$ since the Dobbertin exponent can be defined for even as well as odd dimensions. Nonetheless, by Lemma \ref{lem:gcd2}, it suffices to consider $l = \wt(D_t)$.

\begin{theorem}
    Let $t > 2$ be a natural number and $n = 5t$. Then the Dobbertin exponent $D_t=2^{4t} + 2^{3t} + 2^{2t} + 2^t - 1$ over $\F_{2^n}$, can never be cyclotomic equivalent to $e(l, k)$ for $t > 2$ and any $k$ with $\gcd(k,n) \le 2$.
  \label{propDobbertin}
\end{theorem}
\begin{proof}
Let us assume that the Dobbertin exponent is cyclotomic equivalent to $e(t+3,i)$ for some $i$. Then there exists some $a < n$ such that
\begin{equation*}
  2^{4t + a} + 2^{3t + a} + 2^{2t + a} + 2^{t + a} - 2^a \equiv e(t+3,i) \pmod{n}.
\end{equation*}
As before, we divide the proof into several cases depending on which of the exponents of the terms of the Dobbertin exponent need to be modulated.

\noindent\textbf{Case 1:} If $4t + a < n$, then no modulation is necessary, and the Dobbertin exponent is of the form
\begin{equation*}
  2^{4t + a} + 2^{3t + a} + 2^{2t + a} + \sum_{j = a}^{t+a-1} 2^j,
\end{equation*}
giving the set of exponents
\begin{equation*}
  A = \{ a, a +1, a+2, \dots, a + t - 1 \} \cup \{ 2t + a, 3t + a, 4t + a \}.
\end{equation*}
The set of exponents of $e(t+3,i)$ is of course
\begin{equation*}
  B = \{ 0, i, 2i, \dots, (t+3)i \} \Mod{n}.
\end{equation*}
Since $0 \in B$, then $0 \in A$, and we can only have $a = 0$ unless $t = 0$. Thus, the set $A$ becomes
\begin{equation*}
  A = \{ 0, 1, 2, \dots, t-1 \} \cup \{ 2t, 3t, 4t \}.
\end{equation*}
Since $A \equiv B \pmod{n}$, we must have some $1 \le \alpha, \beta \le t+3$ such that $\alpha i \equiv 1 \pmod{n}$ and $\beta i \equiv 4t \pmod{n}$. Now, $(\beta - \alpha)i  \equiv 4t - 1 \pmod{n}$, and $4t - 1$ is clearly not in $A$ unless $t \le 2$. Thus, we must have $\beta - \alpha \notin \{ 0, 1, \dots, t+3 \}$. Since $\alpha, \beta \le t+3$, this can only happen if $\alpha > \beta$. In this case, $\alpha - \beta \in \{ 0, 1, \dots, t+3 \}$, and $(\alpha - \beta)i \equiv 1 - 4t \equiv 1 - 4t + 5t \equiv t + 1 \pmod{n}$, which is also not in $A$. We have thus obtained a contradiction.

\noindent\textbf{Case 2:} If $3t + a < n \le 4t + a$, let $k = 4t + a - n$. The Dobbertin exponent becomes
\begin{equation*}
  2^{3t + a} + 2^{2t + a} + 2^{t + a} - 2^a + 2^k = 2^{3t + a} + 2^{2t + a} + \left( \sum_{j = a}^{t+a-1} 2^j \right) + 2^k,
\end{equation*}
since $k < a$, or, equivalently $4t + a - n < a$, i.e. $4t < n$. The set of exponents is
\begin{equation*}
  A = \{ a, a+1, a+2, \dots, a+t-1 \} \cup \{ k, 2t + a, 3t + a \}.
\end{equation*}
Once again, $0 \in B$, and we can not have $a = 0$ since then $4t + a$ is always less than $n$. Consequently, we must have $k = 0$, i.e. $4t + a = n = 5t$, i.e. $a = t$. The set $A$ becomes
\begin{equation*}
  \{ t, t+1, t+2, \dots, 2t-1 \} \cup \{ 0, 3t, 4t \}.
\end{equation*}
We must have $\alpha, \beta \in \{0,1,\dots,t+3\}$ such that $\alpha i \equiv t+1 \pmod{n}$ and $\beta i \equiv 4t \pmod{n}$. Then $(\beta - \alpha)i \equiv 3t - 1 \pmod{n}$ and $(\alpha - \beta)i \equiv 1 - 3t \equiv 2t + 1 \pmod{n}$, with either $\alpha - \beta$ or $\beta - \alpha$ being in $\{ 0,1,2,\dots,t+3\}$, and neither of $2t + 1$ and $3t - 1$ being in $A$. We have thus reached a contradiction.

\noindent\textbf{Case 3:} If $2t + a < n < 3t + a$, then let $k = 3t + a - n$. The exponent $4t + a = 3t + a + t$ is congruent with $k + t$ modulo $n$, and $k + t < n$ since $k + t = 4t + a - n < n$, i.e. $a < 2n - 4t = n + t$. The Dobbertin exponent thus becomes
\begin{equation*}
  2^{a + 2t} + 2^{a+t} - 2^a + 2^{k+t} + 2^k,
\end{equation*}
giving the set of exponents
\begin{equation*}
  A = \{ a, a+1, a+2, \dots, a+t-1 \} \cup \{ k, k+t, a+2t \}.
\end{equation*}
The only possible element in $A$ that can be equal to $0 \in B$ is $k = 0$, so that $a + 3t = n = 5t$, i.e. $a = 2t$. Then the set $A$ becomes
\begin{equation*}
  A = \{ 2t, 2t+1, 2t+2, \dots, 3t-1 \} \cup \{ 0, 3t, 4t \}.
\end{equation*}
Now, $\alpha i \equiv 2t + 1 \pmod{n}$ and $\beta i \equiv 3t \pmod{n}$ for some $1 \le \alpha, \beta \le t+3$, and so $(\beta - \alpha)i \equiv t-1$, which is not in $A$; $(\alpha - \beta)i \equiv 1-t \equiv 4t - 1$ which is also not in $A$, and since one of $\alpha - \beta$ and $\beta - \alpha$ must lie in $\{ 0, 1, 2, \dots, t+3 \}$, we have reached a contradiction.

\noindent\textbf{Case 4:} If $a+t < n < a + 2t$, then let $k = a + 2t - n$. As before, the Dobbertin exponent modulo $n$ becomes
\begin{equation*}
  2^{a + t} - 2^a + 2^{k + 2t} + 2^{k+t} + 2^k,
\end{equation*}
and it is easy to verify that $k + 2t < n$ and that $k + 2t < a$. The exponent thus becomes
\begin{equation*}
  \left( \sum_{j = a}^{a+t-1} 2^j \right) + 2^{k + 2t} + 2^{k+t} + 2^k,
\end{equation*}
and hence
\begin{equation*}
  A = \{ a, a+1, a+2, \dots, a+t-1 \} \cup \{ k, k+t, k+2t \}.
\end{equation*}
As before, we must necessarily have $k = 0$, so we get $a + 2t = n = 5t$, i.e. $a = 3t$. The set $A$ becomes
\begin{equation*}
  A = \{ a, a+1, \dots, a+t-1 \} \cup \{ 0, t, 2t \}.
\end{equation*}
Taking $\alpha$ and $\beta$ such that $\alpha i \equiv 3t + 1 \pmod{n}$ and $\beta i \equiv t \pmod{n}$ then leads to a contradiction as before.

\noindent\textbf{Case 5:} If $a + t > n$, then let $k = a + t - n$. We can see that $4t + a \equiv 3t + k \pmod{n}$, and $3t + k < n$, so that the exponent becomes
\begin{equation*}
  -2^a + 2^{3t + k} + 2^{2t + k} + 2^{t + k} + 2^k = -(2^a - 2^{3t + k}) + 2^{2t + k} + 2^{t + k} + 2^k.
\end{equation*}
Using Observation~\ref{obsBinaryComplement}, the above becomes
\begin{equation*}
  \left( \sum_{j = 0}^{3t + k - 1} 2^j \right) + \left( \sum_{j = a}^{n-1} 2^j \right) + 2^{2t + k} + 2^{t + k} + 2^k.
\end{equation*}
This becomes
\begin{equation*}
  \left( \sum_{j = 0}^{k-1} 2^j \right) + 2^{3t + k} + \left( \sum_{j = a}^{n-1} 2^j \right) + 2^{2t +k} + 2^{t + k},
\end{equation*}
giving the set of exponents
\begin{equation*}
  A = \{ 0,1, \dots, k-1 \} \cup \{ a, a+1, \dots, n-1 \} \cup \{ k + t, k + 2t, k + 3t \}.
\end{equation*}
We now find $\alpha, \beta$ in $\{ 1, 2, \dots, t+3 \}$ such that $\alpha i \equiv 1 \pmod{n}$ and $\beta i \equiv n-1 \pmod{n}$ so that $(\alpha + \beta)i \equiv 0 \pmod{n}$. 
Assuming $\gcd(i,n) = 1$ this implies $n \mid \alpha + \beta$, which is impossible. If $\gcd(i, n) = 2$, then either $n \mid \alpha + \beta$, or $n \mid 2(\alpha + \beta)$ implying $n = 5t < 4t + 12$, i.e. $t < 12$, and it can be verified computationally that no equivalence is possible in this case except for $t = 1$.

The proof is done.
\end{proof}

We note that $e(4,4) \equiv 29 \pmod{2^5-1}$ and $e(9,2) \equiv 426 \pmod{2^{10}-1}$, and these are the only two cases in which the Dobbertin exponent can be cyclotomic equivalent to $e(l,k)$.

One can attempt to treat the inverse Dobbertin exponent with the same method that we have used for the other families. Using~\cite[Lemma 9]{BCCDK22}, since the inverse Dobbertin exponent $D_t^{-1}$ is in the cyclotomic coset of $\frac{2^t+1}{2^{2t}+2^t+1}$ modulo $2^{5t}-1$, 
it would be sufficient to investigate the congruence
\[
(2^{2t}+2^t+1)e(l,k)\equiv 2^a (2^t+1)\pmod {2^{5t}-1}.
\]
However, applying Lemma~\ref{lemUniqueExpansion} requires a very large number of degenerate cases to be treated (significantly more that even in the proof of Theorem \ref{propKasami_inv}) and would require multiple pages just to write down. On the other hand, the cases that such a proof would handle on top of our other theorems and characterizations is quite modest: Dobbertin exponents only exist for $n$ that are multiples of $5$, while inverses only exist for odd $n$; and so, the only dimensions $n$ that the lack of such a proof would leave untouched are the odd integers $n$ divisible by $5$. In ordet to handle this subset of dimensions from the point of view of searching for new APN monomials in practice, we simply ran some computer experiments using SageMath on an i7 MacOS with 16GB of RAM. For $n\leq 200$ and $k,l\leq n-1$, the only possible value of $t$ for which we can have equivalence is $t=1$, when $D_1=29\equiv -2\pmod {2^5-1}$, so $D_1^{-1} \equiv 15\pmod {2^5-1}=e(1,4)$. We note that the interval $n \le 200$ should all dimensions $n$ where the investigation of APN-ness can be performed computationally using the currently available methods and computational resources.

\section{Computationally testing APN-ness for the 0-APN exponents}
\label{secComputations}
 We ran experiments to check whether the monomials of the form $x^{e(l, i)}$ with $3 \le l \le 9$ and $1 \le k \le 8$ are APN over the field $\F_{2^n}$ for $2 \le n \le 100$. Our results are presented in Table \ref{tab:apn_2_100} where the entries list the dimensions $n$ for which $e(l, k)$ is an APN exponent over $\F_{2^n}$. The experiments were run on a server with around 500 GB of RAM and 55 Intel Xeon E5-2690 CPU's using the \textit{Magma} Computational Algebra System \cite{bosma1997magma}. We tested APN-ness by checking whether the polynomial $F(x) = x^{e(l,i)} + a^{e(l,i)} + (x+a+1)^{e(l,i)} + 1$ has more than two roots for all possible choices of $a \in \F_{2^n} \setminus \{ 0, a \}$. In our table, the empty cells denote experiments on monomials which did not finish due to time or memory constraints. All APN monomials encountered in our experiment were cyclotomic equivalent to representatives from the known families found in Table~\ref{tableMonomials}.
 
 From the table we can observe that the difficulty of testing APN-ness for an exponent $d$ over $\F_{2^n}$ grows not only with the dimension $n$, but also with $d$ itself. Indeed, for small values of $d = e(l,i)$, we were able to test APN-ness for all $n \le 100$, while for values as small as $e(6,6) \approx 10^9$, this was no longer possible. This also illustrates the advantage of using 0-APN monomials as an intermediate step in the search of APN monomials, since it is significantly easier to characterize 0-APN-ness or to test it computationally.
 
 \begin{table}[h]
\footnotesize
    \centering
    \begin{tabular}{|c|c|c|c|c|c|c|c|c|}
    \hline
    \diagbox[]{$l$}{$i$} & 1 & 2 & 3 & 4 & 5 & 6 & 7 & 8 \\
    \hline
    3 & 5 & 2,4,5 & 3,5 & 2,4,5,8 & 5 & 2,3,4,5,6,12 & 5,7 & 2,4,5,8,16 \\
    \hline
    4 & 2,5,7 & 5,7 & 2,5,7 & 5,7 & 2,7 & 5,7 & 2,5 & \\
    \hline
    5 & 3,9 & 3,9 & 3 & 3,9 & 3,5,9 & & & \\
    \hline
    6 & 2,4,7,11 & 2,7,11 & 2,3,4,7,11 & 2,4,7,11 & & & & \\
    \hline
    7 & 5,13 & 5,13 & 5,13 & & & & & \\
    \hline
    8 & 2,3,5,6,9,15 & 3,5,9,15 & 2,5 & & & & & \\
    \hline
    9 & 5,7,17 & 2,4,5,7,10,17 & & & & & & \\
    \hline
    \end{tabular}
    \caption{Dimensions $n = 2$ to $100$ where $x^{e(l, i)}$ is APN over $\F_{2^n}$.}
    \label{tab:apn_2_100}
\end{table}

\section{Conclusion}

We introduced an infinite class of exponents $e(l,k)$ with two parameters $l,k \in \mathbb{N}$, and showed how to easily find infinitely many dimensions $n$ for which $x^{e(l,k)}$ is 0-APN for any choice of $l$ and $k$. We discussed how our theoretical results can be extended with the help of some computations in order to characterize the set of all dimensions $n$ over which $x^{e(l,k)}$ is 0-APN. The introduced class of exponents is significantly more tractable (both from a computational and a mathematical point of view) than in general, and provides a promising set of exponents that may lead to new APN monomials.

Taking advantage of this tractable structure of the exponents $e(l,k)$, we characterized precisely when they are cyclotomic equivalent to the known APN families (except in the case of the inverse of the Dobbertin exponents where the proof using our current methods is too technical, so we provided computational data for $n \le 200$ instead). We observed that the Gold functions, their inverses, and the inverse APN function can all be expressed in the form $e(l,k)$ for suitable choices of $l$ and $k$.

We have also provided some computational data on the APN-ness of this class of exponents and outlined the limits of the available computational equipment when it comes to verifying APN-ness.

We hope that further investigations into the structure and properties of the exponents $e(l,k)$ can provide us with additional conditions that might help us to overcome these technical limitations, and to potentially identify new instances of APN monomials over finite fields of large extension degree.

\section*{Acknowledgements}

We would like to thank the editor and the referees for their helpful comments and advice which has helped us to greatly improve the quality of the paper.
The paper was started during an enjoyable visit of P. S. at the Selmer Center of University of Bergen in Spring of 2022. He would like to thank the institution for the invitation and the excellent working conditions.

\end{document}